\documentclass[11pt]{amsart}

\usepackage{amsfonts}
\usepackage{amsmath}
\usepackage{setspace}
\usepackage{amsthm}
\usepackage{amssymb}
\usepackage[all]{xy}
\usepackage{color}
\usepackage{mathrsfs}
\usepackage[plain]{fancyref}
\usepackage{fancyhdr}
\usepackage{colortbl}
\usepackage{tikz}
\usetikzlibrary{arrows.meta}
\usetikzlibrary{patterns}
\usepackage[margin=1in]{geometry}

\definecolor{job}{RGB}{200,65,0}

\newcommand{\Z}{{\mathbb{Z}}}

\newcommand{\R}{{\mathbb{R}}}
\newcommand{\N}{{\mathbb{N}}}

\newcommand{\pushout}{\ar@{}[ul(0.35)]|-{\ulcorner}}
\newcommand{\pullback}{\ar@{}[dr(0.35)]|-{\lrcorner}}

\DeclareMathOperator{\Rep}{Rep}
\DeclareMathOperator{\rep}{rep}
\newcommand{\repAR}{\rep_k(A_\R)}
\newcommand{\RepAR}{\Rep_k(A_\R)}

\newcommand{\RepbAR}{\Rep_k^{\rm{b}}(A_\R)}

\newcommand{\ReppwfAR}{\Rep_k^{\rm{pwf}}(A_\R)}

\DeclareMathOperator{\Hom}{Hom}

\definecolor{purple}{RGB}{155,0,155}

\definecolor{gordanagreen}{RGB}{0,155,0}

\definecolor{readableyellow}{RGB}{200,200,0}

\newcommand{\im}{\mathop{\text{im}}}

\newcommand{\coker}{\mathop{\text{coker}}}
\newcommand{\supp}{\mathop{\text{supp}}}
\newcommand{\Ext}{\mathop{\text{Ext}}}
\newcommand{\End}{{\mathop{\text{End}}}}

\usepackage{hyperref}

\newtheorem{lemma}{Lemma}[subsection]
\newtheorem{proposition}[lemma]{Proposition}
\newtheorem{theorem}[lemma]{Theorem}
\newtheorem{cor}[lemma]{Corollary}

\newtheorem{thm}{Theorem}

\theoremstyle{definition}
\newtheorem{definition}[lemma]{Definition}
\newtheorem{remark}[lemma]{Remark}
\newtheorem{example}[lemma]{Example}
\newtheorem{deff}[lemma]{Definition}
\newtheorem{notation}[lemma]{Notation}

\newtheorem{construction}[lemma]{Construction}

\title[Continuous Quivers of Type $A$ (I)]{Continuous Quivers of Type $A$ (I)\\Foundations}
\author[Igusa]{Kiyoshi Igusa}
\thanks{First author supported by the Simons Foundation}

\author[Rock]{Job D.~Rock}
\author[Todorov]{Gordana Todorov}
\date{\today}

\begin{document}
\maketitle


\begin{abstract}
We generalize type $A$ quivers to continuous type $A$ quivers and prove initial results about pointwise finite-dimensional (pwf) representations.
We classify the indecomosable pwf representations and provide a decomposition theorem, recovering results of Botnan and Crawley-Boevey \cite{Crawley-Boevey2015, BotnanCrawley-Boevey}.
We also classify the indecomposable pwf projective representations.
Finally, we prove that many of the properties of finite-dimensional type $A_n$ representations are present in finitely generated pwf representations.
This is the self-contained foundational part of a series of works to study a generalization of continuous clusters categories and their relationship to other type $A$ cluster structures.
\end{abstract}
\tableofcontents

\section*{Introduction}
\subsection*{History}
The indecomposable, finite-dimensional representations of type $A$ quivers and were classified by Gabriel in \cite{Gabriel1972}.
In particular this yielded an understanding of all finite-dimensional representations of type $A$ quivers.

Representations of quivers, and in particular type $A$ quivers, have been used extensively in persistent homology.
Persistent homology has recently been used to study fractal dimension \cite{PersistentDimension1, PersistentDimension2} and has been shown to be effective in recovering some signals in noise \cite{PersistentDimension3}.
Persistent homology has been applied to 3D shape classifications \cite{3Dshapes}, the study of plant root systems \cite{plantroots}, identification of breast cancer subtypes \cite{breastcancer}, and many other real world applications.

Representations of $\R$ and of the infinite zigzag are generalizations of type $A_n$ quiver representations
The first decomposition theorem for representations of $\R$ was proved by Crawley-Boevey in \cite{Crawley-Boevey2015}.
It states that every pointwise finite representation of $\R$ is a sum of indecomposable representations which are supported on intervals in $\R$. 
For general representations of $\R$ the support intervals can be any interval: $(a,b),(a,b],[a,b)$ or $[a,b]$. 
Carlsson, de Silva, and Mozorov introduced zigzag persistent homology in \cite{CarlssondeSilvaMorozov2009} and Botnan proved a similar decomposition theorem to Crawley-Boevey's for infinite zigzag persistence in \cite{Botnan}.

Representations of quivers have also been used to understand cluster algebras via the construction of cluster categories.
Cluster algebras were introduced by Fomin and Zelevinsky in order to better understand scattering diagrams in particle physics \cite{ClusterAlgebrasI}.
Cluster algebras come equipped with a set of \emph{cluster variables}, sets of cluster variables called \emph{clusters}, and a \emph{mutation} process to move from one cluster to another.
Buan, Marsh, Reineke, Rietein, and the third author constructed a cluster category whose indecomposable objects correspond to cluster variables, maximally rigid sums of indecomposables correspond to clusters, and mutation of clusters was given by homological approximations.
In particular, a type $A_n$ cluster algebra can be studied via the cluster category built from finite-dimensional type $A_n$ representations.
The first and third author generalized this construction to a continuous version in \cite{IgusaTodorov2015}.
 
\subsection*{Contributions}
We generalize type $A_n$ quivers to continuous quivers of type $A$ and study their representations. These generalize representations of the real line which are the basis for the continuous cluster category of \cite{IgusaTodorov2015}.
The present paper is a self-contained foundational paper with a focus on representation theoretic techniques.
Our goal is to study continuous quivers of type $A$, representations of such quivers, a generalization of the continuous cluster category, and what these continuous constructions tell us about the corresponding constructions for $A_n$.

We first consider an alternating orientation on $\R$ given by a discrete subset $S=\{\cdots <s_k<s_{k+1}<\cdots\}\subset \R$ and a partial ordering $\preceq$ on $\mathbb R$ given by $x\prec y$ if $s_{2k}\le x<y\le s_{2k+1}$ for some $k$ or if $s_{2k-1}\leq y<x\leq s_{2k}$. The elements $x,y\in\R$ are not related if there is an element of $S$ in the open interval $(x,y)$. This is the continuous version of the zig-zag which is the quiver with vertex set $\mathbb Z$ with one arrow either $i\to i+1$ or $i+1\to i$ between successive integers (see \cite{ZomorodianCarlsson}). 

Let $A_\R$ denote the real line with alternating orientation given by a subset $S$. For any interval, i.e. connected subset, $I\subseteq\mathbb R$, we will construct a pointwise one dimensional representation $M_I$ with support equal to $I$, called interval indecomposable representations.~(See Definition \ref{def: MI}.)
The first theorem takes two representations of $A_\R$ known to be indecomposable (Proposition \ref{prop:sufficientProp}) and tells us when they are isomorphic.
We allow for any alternating orientation so long as $S$ does not have accumulation points and provide a theorem about indecomposable pwf representations and the decomposition of pwf representations: Theormes \ref{thm:intro interval} and \ref{thm:intro decomposition}, respectively.
In Section \ref{sec:relation to BC-B} we discus the relationship between Theorem \ref{thm:intro decomposition}, decomposition results in \cite{Botnan, BotnanCrawley-Boevey, Crawley-Boevey2015}, and the choice of method of proof in the present paper.

\begin{thm}[Theorem \ref{thm:iso-indecomps}]\label{thm:intro interval}
The representations $M_I$ are indecomposable and any pointwise one-dimensional indecomposable representation of $A_\R$ is isomorphic to $M_I$ for some interval $I\subseteq\R$.
Let $V$ and $V'$ be indecomposable representations of a continuous type $A$ quiver.
Then $V\cong V'$ if and only if $\supp V=\supp V'$.
\end{thm}
\begin{thm}[Theorem \ref{thm:indecomposables}]\label{thm:intro decomposition}
Let $V$ be a pointwise finite-dimensional representation of a continuous type $A$ quiver.
Then $V$ is a direct sum of interval indecomposables.
\end{thm}

The proofs of Theorems \ref{thm:intro interval} and \ref{thm:intro decomposition} use Theorems \ref{thm:point projective} and \ref{thm:characterization of one sided projectives}, which classify the interval projective representations in the category of pointwise finite-dimesional representations, denoted $\ReppwfAR$.
In particular we provide this characterization before we prove our decomposition theorem.
Combined with Theorem \ref{thm:intro decomposition}, one obtains a complete description of indecomposable projective objects in $\ReppwfAR$ as Theorem \ref{thm:intro projectives}.
\begin{thm}[Theorems \ref{thm:point projective}, \ref{thm:characterization of one sided projectives}, and \ref{thm:GeneralizedBarCode} and Remark \ref{rem:indecomposableprojectives}]\label{thm:intro projectives}
Any indecomposable projective object in $\ReppwfAR$ is isomorphic to one of the following, where $a$ may be $\pm\infty$.
\begin{enumerate}
\item $P_a$ (for $-\infty < a <+\infty$) given by
\begin{align*}
P_a(x) &= \left\{\begin{array}{ll} k & x\preceq a \\ 0 & \text{otherwise} \end{array}\right. & 
P_a(x,y) &= \left\{\begin{array}{ll} 1_k & y\preceq x \preceq a \\ 0 & \text{otherwise}. \end{array}\right.
\end{align*} 
\item $P_{a)}$ given by
\begin{align*} P_{a)} &= \left\{ \begin{array}{ll}k & x\preceq a, x<a \\ 0 & \text{otherwise} \end{array}\right. & 
P_{a)}(x,y) &= \left\{\begin{array}{ll} 1_k & y\preceq x \preceq a, y\leq x < a \\ 0 & \text{otherwise}. \end{array}\right. \end{align*}
\item $P_{(a}$ given by
\begin{align*} P_{(a} &= \left\{ \begin{array}{ll}k & x\preceq a, a<x \\ 0 & \text{otherwise} \end{array}\right. & 
P_{(a}(x,y) &= \left\{\begin{array}{ll} 1_k & y\preceq x \preceq a, a<x\leq y \\ 0 & \text{otherwise}. \end{array}\right.\end{align*}
\end{enumerate} 

\end{thm}

In Section \ref{sec:little rep} we prove properties about the category of finitely generated representations (Definition \ref{def:finitelygeneratedreps}) over any continuous quiver of type $A$, denoted $\repAR$.
In the $A_n$ case, finitely generated and finite-dimensional representations coincide.
Since only finite sums of simple representations of a continuous quiver are finite-dimensional, we instead consider finitely generated representations in $\ReppwfAR$.
Theorem \ref{thm:intro little rep} highlights some similarities and differences between $\rep(A_n)$ and $\repAR$.
Some of the properties extended to pointwise finite-dimensional representations and bounded-dimensional representations (Definition \ref{def:smallerreps}), denoted $\RepbAR$.
\begin{thm}[Theorem \ref{thm:little rep}]\label{thm:intro little rep}
Let $A_\R$ be a continuous quiver of type $A$.
Then the following hold.
\begin{enumerate}
\item For indecomposable representations $M_I$ and $M_J$ in $\ReppwfAR$, $\RepbAR$, or $\repAR$, we have $\Hom(M_I,M_J)\cong k$ or $\Hom(M_I,M_J)=0$ (Proposition \ref{prop:homiskor0}).
\item Every morphism $f:V\to W$ in $\ReppwfAR$, $\RepbAR$, or $\repAR$ has a kernel, a cokernel, and coinciding image and coimage in that category. (Lemma \ref{lem:abelian})
\item The category $\repAR$ Krull-Schmidt, \emph{but not} artinian (Lemma \ref{lem:krull-schmidt}, Proposition \ref{prop:not artinian}).
\item The global dimension of $\repAR$ is 1 (Proposition \ref{prop:projectiveresolution}).
\item The $\Ext$ space of two indecomposables $M_I$ and $M_J$ in $\ReppwfAR$, $\RepbAR$, or $\repAR$ is either isomorphic to $k$ or is 0 (Proposition \ref{prop:smallext}).
\item While some Auslander--Reiten sequences exist (Proposition \ref{prop:ARexistence}), some indecomposables have \emph{neither} a left \emph{nor} a right Auslander--Reiten sequence (Proposition \ref{prop:no AR sequences}).
\end{enumerate}
\end{thm}


\subsection*{Acknowledgements}
The authors would like to thank Ralf Schiffler for organizing the Cluster Algebra School at the University of Connecticut and Shijie Zhu for helpful discussions.
They would also like to thank Magnus B.~Botnan, Bill Crawley-Boevey, Bernhard Keller, and Francesco Sala for references to related work.
The second author would also like to thank Eric Hanson for helpful discussions.

\section{Continuous Quivers of Type $A$}

\noindent We let $k$ denote a field for the entirety of this paper.

\subsection{Quiver of Continuous Type $A$: $A_\R$}
The goal of this section is to generalize the definition of type $A$ quivers to a continuous setting.
The set $\R$ will serve as the vertices in our quiver.
We will choose a set of sinks and sources, which will induce the orientation on the continuous quiver by indicating which vertices have paths to which others.
The picture below gives an intuitive idea of the result of choosing a continuous type $A$ quiver and the definition follows.
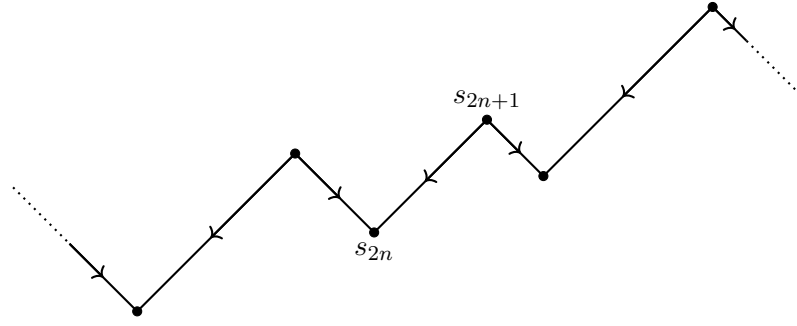
\begin{displaymath}\begin{tikzpicture}[scale=1.5]
\draw[thick, dotted] (-3.5,0.4) -- (-3,-0.1);
\draw[thick, dotted] (3,1.7) -- (3.5,1.2);
\draw[thick] (-3,-0.1) -- (-2.4,-0.7) -- (-1,0.7) -- (-0.3,0) -- (0.7,1) -- (1.2,0.5) -- (2.7,2) -- (3,1.7);
\draw[fill=black] (-2.4,-0.7) circle[radius=0.4mm];
\draw[fill=black] (-1,0.7) circle[radius=0.4mm];
\draw[fill=black] (-0.3,0) circle[radius=0.4mm];
\draw[fill=black] (0.7,1) circle[radius=0.4mm];
\draw[fill=black] (1.2,0.5) circle[radius=0.4mm];
\draw[fill=black] (2.7,2) circle[radius=0.4mm];
\draw[thick, ->] (-3,-0.1) -- (-2.7,-0.4);
\draw[thick, ->] (-1,0.7) -- (-1.75,-0.05);
\draw[thick, ->] (-1,0.7) -- (-0.6,0.3);
\draw[thick, ->] (0.7,1) -- (0.15,0.45);
\draw[thick, ->] (0.7,1) -- (1,0.7);
\draw[thick, ->] (2.7,2) -- (1.9,1.2);
\draw[thick, ->] (2.7,2) -- (2.9,1.8);
\draw(-0.3,0) node[anchor=north] {$s_{2n}$};
\draw(0.7,1) node[anchor=south] {$s_{2n+1}$};
\end{tikzpicture}\end{displaymath}

\begin{definition}\label{def:AR}A \underline{quiver of continuous type $A$}, denoted by $A_\R$, is a triple $(\R,S,\preceq)$, where:
\begin{enumerate}
\item 
\begin{enumerate}
\item$S\subset \R$ is a discrete subset, possibly empty, with no accumulation points.
\item Order on $S\cup\{\pm\infty\}$ is induced by the order of $\R$, and $-\infty<s<+\infty$ for $\forall s\in S$.
\item Elements of $S\cup\{\pm\infty\}$ are indexed by a subset of $\Z\cup\{\pm\infty\}$ so that $s_n$ denotes the element of 
$S\cup\{\pm\infty\}$ with index $n$. The indexing must adhere to the following two conditions:
\begin{itemize}
\item[i1] There exists $s_0\in S\cup\{\pm\infty\}$.
\item[i2] If $m\leq n\in\Z\cup\{\pm\infty\}$ and $s_m,s_n\in S\cup\{\pm\infty\}$ then for all $p\in\Z\cup\{\pm\infty\}$ such that $m\leq p \leq n$ the element $s_p$ is in $S\cup\{\pm\infty\}$.
\end{itemize}
\end{enumerate}
\item New partial order $\preceq$ on $\R$, which we call  the \underline{orientation} of $A_\R$, is defined as:
\begin{itemize}
\item[p1\ ] The $\preceq$ order between consecutive elements of $S\cup\{\pm\infty\}$ does not change.
\item[p2\ ] Order reverses at each element of $S$.
\item[p3\ ] If $n$ is even $s_n$ is a sink.
\item[p3'] If $n$ is odd $s_n$ is a source.
\end{itemize}
\end{enumerate}
\end{definition}

\begin{definition}\label{def:ARalgebra}
Let $A_\R=(\R,S, \preceq)$ be a quiver of continuous type $A$. Then 
the associated \underline{continuous path algebra} $kA_\R$ is the associative algebra over $k$ (without unity) whose basis consists of pairs $(x,y)$, where $y\preceq x$.
Multiplication on the pairs is given by
\begin{displaymath} (w,x)(y,z) = \left\{ \begin{array}{ll}(w,z) & x=y \\ 0 & x\neq y. \end{array} \right.\end{displaymath}
\end{definition}

\begin{remark}
The indexing requirements on $S$ have the following immediate consequences.
\begin{itemize}
\item If $S$ is empty then either (i) $s_{-1}=-\infty$ and $s_0=+\infty$ or (ii) $s_0=-\infty$ and $s_1=+\infty$.
\item If $S$ is unbounded above (below) then $+\infty=s_{+\infty}$ ($-\infty=s_{-\infty}$).
\item If $S$ is bounded above (below) then there is no $s_{+\infty}$ ($s_{-\infty}$) in $\bar{S}$.
\end{itemize}

The rules for the partial order have the following consequences.
If $x<y\in\R$ and there is some $s_n\in S$ such that $x<s_n<y$ then $x \npreceq y$ and $y\npreceq x$.
If $x\leq y\in\R$ and there exists $s_n,s_{n+1}\in \bar{S}$ such that $s_n\leq x\leq y\leq s_{n+1}$ then:
\begin{align*}
x\preceq y & \text{ if } n \text{ is even} \\
y \preceq x & \text{ if } n \text{ is odd}.
\end{align*}
\end{remark}

\begin{example}\label{xmp:basicorientations}
We provide four examples of $S$ and the induced partial order $\preceq$ on $\R$.
\begin{enumerate}
\item A finite example: $S=\{\frac{1}{2},\pi\}$, $\bar{S}=\{-\infty,\frac{1}{2},\pi,+\infty\}$, $s_{-2}=-\infty$, $s_{-1}=\frac{1}{2}$, $s_0=\pi$, and $s_1=+\infty$.
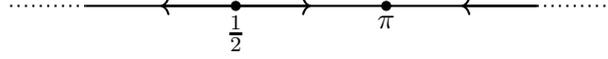
\begin{displaymath}\begin{tikzpicture}[scale=2]
\draw[black, dotted, thick] (-2,0) -- (-1.5,0);
\draw[black, dotted, thick] (1.5,0) -- (2,0);
\draw[black,thick] (-1.5,0) -- (1.5,0);
\draw[black,thick,->] (-.5,0) -- (0,0);
\draw[black,thick,->] (-.5,0) -- (-1,0);
\draw[black,thick,->] (1.5,0) -- (1,0);
\filldraw[fill=black,draw=black] (-.5,0) circle[radius=.3mm];
\filldraw[fill=black,draw=black] (.5,0) circle[radius=.3mm];
\draw (-.5,0) node[anchor=north] {${\frac{1}{2}}$};
\draw (.5,0) node[anchor=north] {${\pi}$};
\end{tikzpicture}\end{displaymath}
\item A ``half'' unbounded example: $S=\{2n: n\in\N\}$, $s_{-1}=-\infty$, $s_n=2n$ when $n\geq 0$, and $s_{+\infty}=+\infty$.
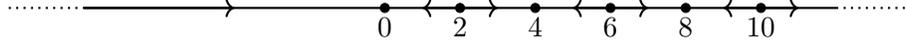
\begin{displaymath}\begin{tikzpicture}
\draw[black, dotted, thick] (-5,0) -- (-4,0);
\draw[black, dotted, thick] (6,0) -- (7,0);
\draw[black, thick] (-4,0) -- (6,0);
\draw[black, thick,->] (-4,0) -- (-2,0);
\foreach \x in {0,1,2}
{
	\draw[black, thick, ->] (2*\x+1,0) -- (2*\x+0.5,0);
	\draw[black,thick, ->] (2*\x+1,0) -- (2*\x + 1.5,0);
	\filldraw[fill=black,draw=black] (2*\x,0) circle [radius=.6mm];
	\filldraw[fill=black,draw=black] (2*\x+1,0) circle [radius=.6mm];
}
\draw (0,0) node[anchor=north] {{0}};
\draw (1,0) node[anchor=north] {{2}};
\draw (2,0) node[anchor=north] {{4}};
\draw (3,0) node[anchor=north] {{6}};
\draw (4,0) node[anchor=north] {{8}};
\draw (5,0) node[anchor=north] {{10}};
\end{tikzpicture}\end{displaymath}
\item An unbounded example: $S=\{\frac{n}{2}: n\in\Z\}$, $s_{-\infty}=-\infty$, $s_n=\frac{n}{2}$, and $s_{+\infty}=+\infty$.
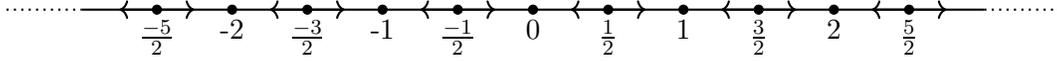
\begin{displaymath}\begin{tikzpicture}
\draw[black, dotted, thick] (-7,0) -- (-6,0);
\draw[black, dotted, thick] (6,0) -- (7,0);
\draw[black, thick] (-6,0) -- (6,0);
\draw[black, thick, ->] (-5,0) -- (-5.5,0);
\draw[black,thick, ->] (-5,0) -- (-4.5,0);
\filldraw[fill=black,draw=black] (-5,0) circle [radius=.6mm];
\foreach \x in {-2,...,2}
{
	\draw[black, thick, ->] (2*\x+1,0) -- (2*\x+0.5,0);
	\draw[black,thick, ->] (2*\x+1,0) -- (2*\x + 1.5,0);
	\filldraw[fill=black,draw=black] (2*\x,0) circle [radius=.6mm];
	\filldraw[fill=black,draw=black] (2*\x+1,0) circle [radius=.6mm];
}
\draw (-5,0) node[anchor=north] {{$\frac{-5}{2}$}};
\draw (-4,0) node[anchor=north] {{-2}};
\draw (-3,0) node[anchor=north] {{$\frac{-3}{2}$}};
\draw (-2,0) node[anchor=north] {{-1}};
\draw (-1,0) node[anchor=north] {{$\frac{-1}{2}$}};
\draw (0,0) node[anchor=north] {{0}};
\draw (1,0) node[anchor=north] {{$\frac{1}{2}$}};
\draw (2,0) node[anchor=north] {{1}};
\draw (3,0) node[anchor=north] {{$\frac{3}{2}$}};
\draw (4,0) node[anchor=north] {{2}};
\draw (5,0) node[anchor=north] {{$\frac{5}{2}$}};
\end{tikzpicture}\end{displaymath}
\item One of the two $S=\emptyset$ possibilities: $S=\emptyset$, $s_0=-\infty$, and $s_1=+\infty$. This causes $\preceq$ to coincide with $\leq$.
\end{enumerate}
\end{example}

\begin{remark}
It is important to note that the choice which element of $S$ becomes $s_0$ determines the entire indexing of $S$ and thus the entire partial order $\preceq$.
Additionally, given a set $\bar{S}$ there are exactly two partial orders $\preceq$ possible no matter which element of $S$ is chosen to be $s_0$.
The two partial orders are opposites of each other.
\end{remark}

\begin{remark}\label{rem:ARconvention}
From now on, whenever we refer to $A_\R$, we are implicitly assuming some $S$ with indexing and $\preceq$ have been set.
By `the straight descending orientation' we mean the one where $S=\emptyset$, $s_0=-\infty$, and $s_1=+\infty$ as in Example \ref{xmp:basicorientations}. This is the case where $\preceq$ coincides with $\leq$.
\end{remark}

\subsection{Representations of $A_\R$: $\RepAR$}
\begin{deff}\label{def:representation}
A \underline{representation} $V$ of $A_\R$ is a module over the path algebra $kA_\R$.
Explicitly, one assigns to each real number $x$ a vector space $V(x)$ and to each pair $(x,y)$, where $y\preceq x$, a linear transformation $V(x,y):V(x)\to V(y)$ such that $V(y,z)\circ V(x,y)=V(x,z)$ whenever such a composition is defined.
The \underline{support} of a representation $V$ is the set of all $x\in\R$ such that $V(x)\neq 0$.
We denote the support of a representation $V$ by $\supp V$.

A \underline{simple representation at $x$} is a representation $V$ such that $V(x)\cong k$ and if $y\neq x$ then $V(y)=0$.
The linear map $V(x,x)$ is the identity and $V(y,z)=0$ if $y\neq x$ or $z\neq x$.
\end{deff}
\begin{deff}\label{def:morphism}
A \underline{morphism} $f:V\to W$ of representations of $A_\R$ is a morphism of $kA_\R$ modules.
Explicitly, it is a collection of linear maps $f(x):V(x)\to W(x)$, for all $x\in\R$, making the following squares commute for each pair $x,y\in\R$ where $y\preceq x$:
\begin{displaymath}
\xymatrix{ V(x) \ar[r]^-{V(x,y)} \ar[d]_-{f(x)} & V(y) \ar[d]^-{f(y)} \\ W(x) \ar[r]_-{W(x,y)} & W(y). }
\end{displaymath}
\end{deff}

Since we're working with modules over an associative algebra, and associative algebras are in particular rings, the category of $k$-representations of $A_\R$, denoted $\RepAR$, is abelian.

Propositions \ref{prop:isomorphisms} and \ref{prop:isomorphic implies same support} can be proved almost the exact same way as they would for discrete quivers of type $A$.
\begin{proposition}\label{prop:isomorphisms}
A morphism of representations $f:V\to W$ in $\RepAR$ is an isomorphism if and only if $f(x)$ is an isomorphism for each $x\in \R$.\end{proposition}

\begin{proposition}\label{prop:isomorphic implies same support}
Let $V$ and $V'$ be representations of $A_\R$ such that $V\cong V'$.
Then $\supp V=\supp V'$.
\end{proposition}

\subsection{The Subcategories $\ReppwfAR$ and $\RepbAR$}

In this subsection we define the pointwise finite and bounded subcategories of $\RepAR$.
We provide examples of representations in each subcategory and highlight the differences between them.

\begin{definition}\label{def:smallerreps}
The category of \underline{pointwise finite representations}, denoted $\ReppwfAR$, is the full subcategory of $\RepAR$ consisting of representations $V$ such that for all $x\in\R$, $\dim V(x)<\infty$.

The category of \underline{bounded representations}, denoted $\RepbAR$, is the full subcategory of $\ReppwfAR$ whose objects are representations $V$ such that there exists $n\in\N$ and for all $x\in\R$, $\dim V(x)<n$.
\end{definition}

It is important to note that the conditions in Definition \ref{def:smallerreps} are not related to the \emph{support} of any representation. I.e. there exist representations in both $\ReppwfAR$ and $\RepbAR$ with unbounded support.
Such examples are provided below.

\begin{example}\label{xmp:boundedreps}
We now give some examples of representations in $\ReppwfAR$ and $\RepbAR$.
Each representation will be over $A_\R$ with the straight descending orientation (see Remark \ref{rem:ARconvention}).
\begin{enumerate}
\item We give an example of a representation in $\RepbAR$ with unbounded support.
A representation in $\RepbAR$ is $V$:
\begin{align*}
V(x) &= \left\{\begin{array}{ll} k & x\geq 0\\ 0 & x<0 \end{array}\right. & V(x,y) & =\left\{\begin{array}{ll} 1_k & 0\leq y\leq x \\ 0 & \text{otherwise}\end{array}\right. 
\end{align*}
Notice that the \emph{support} of $V$ is unbounded.
This is fine. The \emph{dimension} of all the $V(x)$ vector spaces is bounded above by 1.

\item We now give an example of an infinite coproduct that is still in $\RepbAR$. Let $M=\bigoplus M_{\{z\}}$ where for each $z\in \R$, $M_{\{z\}}$ be the following representation of $A_\R$:
\begin{align*}
M_{\{z\}}(x) & = \left\{\begin{array}{ll} k & x=z\\ 0 & \text{otherwise} \end{array}\right. & M_{\{z\}}(x,y) &
=\left\{\begin{array}{ll} 1_k & x=y=z \\ 0 & \text{otherwise}\end{array}\right.
\end{align*}
That is, $M_{\{z\}}$ is the simple representation at $z$, which is in $\RepbAR$.
However, $M=\bigoplus_{z\in\R} M_{\{z\}}$ is also in $\RepbAR$ since $\dim M(x)=1$ for all $x\in\R$.

\item We now give an example of a representation in $\ReppwfAR$ \emph{but not} in $\RepbAR$.
Let $W$ be the representation of $A_\R$ where $W(x)$ is $k^n$ where $n=0$ if $x<1$ and $n$ is the largest integer less than or equal to $x$ otherwise. I.e., $W(10.4)=k^{10}$.
Let $W(x,y)$ be 0 if $y<1$ or $x<1$.
Otherwise, $W(x,y)$ is the projection of the first $\dim W(y)$ coordinates of $k^{\dim W(x)}$ using the standard basis.
For example, $W(10,4)$ is the projection of $k^{10}$ onto the first 4 coordinates.
While $W(x)$ is finite-dimensional for all $x\in\R$, there is no $n$ such that $\dim W(x)\leq n$ for all $x\in\R $.
\end{enumerate}
\end{example}

Originally, the authors only attempted to prove a version of Theorem \ref{thm:indecomposables} for $\RepbAR$.
However, it was noted that nearly all the proof techniques relied on finite-dimensional vector spaces, not on the dimension of the vector spaces being bounded.
In the category $\RepbAR$ the authors discovered projective indecomposable objects that are not projective in $\RepAR$.
Further study revealed these objects to also be projective in $\ReppwfAR$.
See Section \ref{sec:projectives} for details on these new projective objects.
These new projectives in $\ReppwfAR$ are necessary to obtain a category of finitely generated representations (Definition \ref{def:finitelygeneratedreps}, denoted $\repAR$) which has all the reasonable properties one could expect from a continuous version of finitely generated representations.

In contrast to the apparent superiority of $\ReppwfAR$, the category is simply too big to even have all projective covers.
While pathological examples of representations without projective covers can be constructed in both $\ReppwfAR$ and $\RepbAR$, the more well-behaved examples of representations without projective covers exist only in $\ReppwfAR$.
See Example \ref{xmp:no projective cover} in Section \ref{sec:projectives}.
Such a representation does not exist in $\RepbAR$ and so this can be considered the first step towards finitely generated representations.

\section{Classification of Indecomposable Pwf Representations}
In this section we provide a complete classification of indecomposable pointwise finite-dimensional represntations of a continuous quiver of type $A$ (Theorem \ref{thm:GeneralizedBarCode}).
We focus on representation theoretic techniques and provide a self-contained approach.
In particular, we characterize projective indecomposables (Theorems \ref{thm:point projective} and \ref{thm:characterization of one sided projectives}) before our decomposition theorem and obtain the completeness of the classification (Remark \ref{rem:indecomposableprojectives}) as a result.
Additionally, our proof of the decomposition theorem is algorithmic.
We discuss related results by Botnan and Crawley-Boevey in Section \ref{sec:relation to BC-B}.

\subsection{Projectives}\label{sec:projectives}
 
We will construct all pointwise finite-dimensional projective representations in the category $\ReppwfAR$. 
There are two types of indecomposable projectives: projectives $P_c$ generated at one point as in Definition \ref{def:projective generated}, which are quite similar to the projectives for finite quivers. These kinds of projectives are actually projective in both $\ReppwfAR$ and in $\RepAR$.
 The new kind of projectives $P_{c)}$ and $P_{(c}$ will be projectives which have half open intervals as supports as in Definition \ref{def:one sided projectives}. These representations are projective in $\ReppwfAR$ but are not projective in $\RepAR$ (see Example \ref{prop:not projective in big category}).
 We start with the case of a projective generated at one point.
 
 \begin{definition}
 Given any point $c\in \mathbb R$ and any vector space $X$ over $k$, let $(PX)_c$ be the representation defined as follows.
\[
	(PX)_c(x)=\begin{cases} X & \text{if }x\preceq c\\
    0& \text{otherwise}
    \end{cases}
\]
and $(PX)_c(x,y)=id_X$ if $y\preceq x\preceq c$.
 \end{definition}
 
 \begin{lemma}\label{lem: PXc is adjoint to restriction to c}
 For any representation $V$ in $\RepAR$ (not necessarily pointwise finite) and any $k$-vector space $X$ we have: 
 \[
 	\Hom((PX)_c,V)=\Hom_k(X,V(c)),
 \]
 i.e., the functor which takes $X$ to $(PX)_c$ is left adjoint to the evaluation functor $V\mapsto V(c)$.
 \end{lemma}
 
 \begin{proof}
 Given any morphism $f:(PX)_c\to V$, let $f_c:(PX)_c(c)=X\to V(c)$ be the restriction of $f$ to the point $c$. Then, for any $x\preceq c$, the commutativity of the diagram:
 \[
\xymatrix{
(PX)_c(c)=X\ar[d]_{f_c}\ar[r]^{id_X} &
	X=(PX)_c(x)\ar[d]^{f_x}\\
V(c) \ar[r]^{V(c,x)}& 
	V(x)
	}
 \]
 forces the map $f_x:(PX)_c(x)\to X(x)$ to be equal to $V(c,x)\circ f_c$. Conversely, any linear map $g:X\to V(c)$ extends to a morphism $\overline g:(PX)_c\to V$ by the same formula ($\overline g(x)=V(c,x)\circ g:(PX)_c(x)=X\to X(x)$).
 \end{proof}
 
 \begin{theorem}\label{thm:point projective}
For any vector space $X$ and any $c\in \mathbb R$, the representation $(PX)_c$ is projective in $\RepAR$.
 \end{theorem}
 
 \begin{proof}
 Let $p:V\to W$ be an epimorphism and let $f:(PX)_c\to W$ be any morphism. Then $p_c:V(c)\to W(c)$ is an epimorphism. So, the linear map $f_c:X\to W(c)$ lifts to a map $g:X\to V(c)$ which, by Lemma \ref{lem: PXc is adjoint to restriction to c}, extends to a morphism $\overline g:(PX)_c\to V$. Since $p\circ \overline g$ and $f:(PX)_c\to W$ agree at $c$, they are equal by Lemma \ref{lem: PXc is adjoint to restriction to c}. So, $(PX)_c$ is projective.
 \end{proof}
 
Note that $(PX)_c$ is indecomposable if and only if $X$ is one-dimensional as it is in the following definition. In this case the indecomposable projective is denoted simply by $P_c$.
 
 \begin{definition}\label{def:projective generated}
Let  $c\in \mathbb R$. Define the representation $P_c$ in $\RepAR$ as:
 \[\begin{array}{ccc}
	P_c(x)=\begin{cases} k & \text{if }x\preceq c\\
    0& \text{otherwise}
    \end{cases}&,\qquad \qquad&P_c(x,y)=\left\{\begin{array}{ll} 1_k & \text{if } y\preceq x\preceq c \\ 0 &\text{otherwise} \end{array}\right.
    \end{array}
\]
  \end{definition}
 
 The rest of this subsection is devoted to the construction of all pointwise finite-dimensional projective representation, including objects $P_{(a}$ and  $P_{b)}$ for $s_{2n-1}<a<s_{2n}<b<s_{2n+1}$ with supports $(a,s_{2n}]$ and $[s_{2n},b)$ respectively.
 
In order to describe these new types of
projective representations in the category of pointwise finite-dimensional representations of $A_{\mathbb R}$ we need to set up notation of ``image filtration'' (Definition \ref{def:image filtration}) and ``support intervals'' (Definition \ref{def:support intervals}).
Recall $s_n$ is a sink if $n$ is even and a source if $n$ is odd.

\begin{deff}\label{def:image filtration}
Let $V$ be a pointwise finite-dimensional representation of $A_\R$ such that $\supp V\subset [s_0,s_1]$. 
(If $s_0=-\infty$ or $s_1=+\infty$ then $\supp V$ is a subset of $(s_0, s_1]$, $[s_0,s_1)$, or $(s_0,s_1)$, whichever applies.) 
Let $b=s_0$ or $b\in (s_0,s_1)$.

Let $b$ be the greatest lower bound of $\supp V$. When $b\in\supp V$ we set $V^{\bullet}(b)=V(b)$.
The \underline{image filtration} of $V^{\bullet}(b)$ is the set of distinct subspaces of the form $V(x,b)(V(x))$.
Let $V^{\circ}(b)$ be the colimit of the vector spaces $V(x)$, for $b<x<s_1$, with the linear maps $V(x,y)$, for $b<y\leq x<s_1$. Since each $V(x)$ is finite-dimensional, $V^{\circ}(b)$ is at most countably infinite dimensional.
Denote by $V^{\circ}(x,b)$ the colimit linear map from $V(x)$ to $V^{\circ}(b)$.
The \underline{image filtration} of $V^{\circ}(b)$ is the set of distinct (finite-dimensional) subspaces of the form $V^{\circ}(x,b)(V(x))$.

When $b\in \supp V$ we take $I$ to be $[b,c]$ or $[b,c)$.
When $b\notin \supp V$ we take $I$ to be $(b,c]$ or $(b,c)$. 
For all such $I$ and when $b\in\supp V$, let
\begin{equation}\label{eqn:def of VI(b)}
V^{\bullet}_I(b): = \bigcap_{x\in I} V(x,b)(V(x)) \subset V^{\bullet}(b)
\end{equation}
Whether or not $b\in\supp V$, we let
\begin{equation}\label{eqn:def of FVI(b)}
V^{\circ}_{I\setminus \{b\}}(b) :=  \bigcap_{x\in I\setminus \{b\}} V^{\circ}(x,b)(V(x))\subset V^{\circ}(b).
\end{equation}
Then $V^{\bullet}_I(b)$ and $V^{\circ}_I(b)$ are members of the image filtrations of $V^{\bullet}(b)$ and $V^{\circ}(b)$, respectively.
In particular, there exists $x_0$ in $I$ such that $V^{\bullet}_I(b)=V(x_0,b)(V(x_0))$ or $x_0\in I\setminus b$ such that $V^{\circ}_I(b)=V^{\circ}(x_0,b)(V(x_0))$.
Whenever $b\in\supp V$, $V^{\bullet}(b)$ is finite-dimensional and so the image filtration is finite.
Since $V^{\circ}(b)$ may not be finite-dimensional and the dimension of the vector spaces $V(x)$ are not bounded the filtration on $V^{\circ}(b)$ may be infinite but still countable with a minimal term. 
In fact, $V(s_1,b)(b)$ and $V^{\circ}(s_1,b)$ are the minimal objects in the filtrations of $V^{\bullet}(b)$ and $V^{\circ}(b)$, repsectively.~ 
\end{deff}

\begin{remark} \label{rem:V_I(b)}
\begin{itemize}
\item[(a)\ ] If $I=[b,c]$ ($c\in I$) then 
\begin{align*}
V^{\bullet}_I(b)=V^{\bullet}_{[b,c]}(b)&= V(c,b)(V(c)) \\
V^{\circ}_{I\setminus \{b\}}(b) =V^{\circ}_{[b,c]}(b) &=V^{\circ}(c,b)(V(c)).
\end{align*}
\item[(b)\ ] If $I=[b,c)$ ($c\notin I$) then whenever $c\in \R$ we have
\begin{align*} 
V^{\bullet}_I(b)=V^{\bullet}_{[b,c)}(b) & \supset V(c,b)(V(c)) \text{ and} \\ 
V^{\circ}_{I\setminus\{b\}}(b) = V^{\circ}_{(b,c)}(b) & \supset V^{\circ}(c,b)(V(c))
\end{align*}
but in both cases the subspaces may be different.\\
\item[(c)\ ] For any $x<z$ in $I$, the term $V(x,b)(V(x))$ is redundant in the intersection \eqref{eqn:def of VI(b)} since $V(x,b)(V(x))\supset V(z,b)(V(z))$. Thus,
\[
  V^{\bullet}_I(b)=\bigcap_{x\in I} V(x,b)(V(x)) =\bigcap_{y\in I,y\ge z} V(z,b)V(y,z)(V(y))=V(z,b)V^{\bullet}_{I\cap [z,s_1]}(z)
\]
\item[(c')] For any $x<z$ in $I$, the term $V^{\circ}(x,b)(V(x))$ is redundant in the intersection \eqref{eqn:def of FVI(b)} since $V^{\circ}(x,b)(V(x))\supset V^{\circ}(z,b)(V(z))$. Thus,
\[
  V^{\circ}_I(b)=\bigcap_{x\in I} V^{\circ}(x,b)(V(x)) =\bigcap_{y\in I,y\ge z} V^{\circ}(z,b)V^{\circ}(y,z)(V(y))=V^{\circ}(z,b)V^{\bullet}_{I\cap [z,s_1]}(z)
\]
\end{itemize}
\end{remark}

\begin{deff} \label{def:support intervals} Let  $W\subset V^{\bullet}(b)$ be a subspace. Define $I_W$ as: 
$$I_W=\{x\geq b\in \supp V \ |\ W\subset V(x,b)(V(x))\}$$
Such $\{I_W\}$ are called  \underline{support intervals for} $V^{\bullet}(b)$.

Let $W\subset V^{\circ}(b)$ be a finite-dimensional subspace.
Then we define $I_W$ similarly for $b\notin I$:
$$I_W=\{x> b\in \supp V \ |\ W\subset V^{\circ}(x,b)(V(x))\}$$
These $\{I_W\}$ are also called \underline{support intervals for} $V^{\circ}(b)$.
\end{deff}

 \begin{proposition} \label{prop:I_W}
 \begin{itemize}
 \item[(a)\ ] There is a 1-1 correspondence between support intervals for $V^{\bullet}(b)$ and the terms in the image filtration of $V^{\bullet}(b)$ given by $I\mapsto V^{\bullet}_I\subset V^{\bullet}(b)$ and $W\mapsto I_W$.
\item[(a')] There is also a 1-1 correspondence beween the support intervals for $V^{\circ}(b)$ and the terms in the image filtration of $V^{\circ}(b)$ given by $I\mapsto V^{\circ}_I\subset V^{\circ}(b)$ and $W\mapsto I_W$.
\end{itemize}
\end{proposition}
\begin{proof}
We first prove (a).
If $W=V(x_0,b)(V(x_0))$ then $I_W$ contains $x_0$ and $V^{\bullet}_{I_W}$ is the intersection of $V(x_0,b)(V(x_0))=W$ and the subspaces $V(x,b)(V(x))$ which all contains $W$ by definition. So, $V^{\bullet}_{I_W}=W$.

If $I=I_W$ then $W\subset V^{\bullet}_I(b)$ by definition. If $V(x,b)(V(x))$ contains $V^{\bullet}_I(b)$, it contains $W$. So, $I_{V^{\bullet}_I(b)}\subset I_W=I$. But $I\subset I_{V^{\bullet}_I(b)}$. So, $I= I_{V^{\bullet}_I(b)}$ for any support interval $I$.

The proof of (a) as stated works for (a') if we replace $V^{\bullet}_I$ with $V^{\circ}_I$.
\end{proof}

\begin{remark}
The image filtration of $V^{\bullet}(b)$ can be written:
\begin{displaymath}V^{\bullet}(b)\supsetneq V(x_n,b)(V(x_n)) \supsetneq V(x_{n-1},b)(V(x_{n-1})) \supsetneq\cdots\supsetneq V(x_1,b)(V(x_1)).\end{displaymath}
By Proposition Proposition \ref{prop:I_W}, we see this is actually a filtration
\begin{displaymath} V^{\bullet}(b)\supsetneq V^{\bullet}_{I_n}\supsetneq V^{\bullet}_{I_{n-1}}\supsetneq\cdots\supsetneq V^{\bullet}_{I_1} . \tag{$\bullet *$}\end{displaymath}
where each $x_i$ in the first form is an element of $I_i$ in the second form.

For the image filtration of $V^{\circ}(b)$, we have the following equivalent forms, where each $x_i$ in the first form is an element of $I_i$ in the second form:
\begin{displaymath}V^{\circ}(b)\ \cdots\supsetneq V^{\circ}(x_n,b)(V(x_n)) \supsetneq V^{\circ}(x_{n-1},b)(V(x_{n-1})) \supsetneq\cdots\supsetneq V^{\circ}(x_1,b)(V(x_1)).\end{displaymath}
By Proposition \ref{prop:I_W}, we see this is actually a filtration
\begin{displaymath} V^{\circ}(b)\ \cdots\supsetneq V^{\circ}_{I_n}\supsetneq V^{\circ}_{I_{n-1}}\supsetneq\cdots\supsetneq V^{\circ}_{I_1} . \tag{$\circ*$}\end{displaymath}
\end{remark}

\begin{lemma}\label{lem:support of image is support}
\begin{itemize}
\item[(a)\ ] Let $W'\subsetneq W$ be consecutive terms in the image filtration ($\bullet *$) of $V^{\bullet}(b)$. Then $I_W\subsetneq I_{W'}$ and, for any element $x\in (I_{W'}-I_W)$, we have $V(x,b)(V(x))=W'$ in $V^{\bullet}(b)$.
\item[(a')] Let $W'\subsetneq W$ be consecutive terms in the image filtration ($\circ *$) of $V^{\circ}(b)$. Then $I_W\subsetneq I_{W'}$ and, for any element $x\in (I_{W'}-I_W)$, we have $V^{\circ}(x,b)(V(x))=W'$ in $V^{\circ}(b)$.
\end{itemize}
\end{lemma}
\begin{proof}
We prove (a) first.
Since $x\notin I_W$ it follows from Remark \ref{rem:V_I(b)}(c) that  $W$ is not a subset of $V(x,b)(V(x))$  in $V^{\bullet}(b)$. 
Since $x\in I_{W'}$ it follows that $W'\subset V(x,b)(V(x))$. Thus $W'\subset V(x,b)(V(x))\subsetneq W$. Since $W,W'$ are consecutive in the image filtration, $V(x,b)(V(x))=W'$ as claimed.
The proof of (a) as stated works for (a') by replacing $V^{\bullet}_I$ with $V^{\circ}_I$.
 \end{proof}

\begin{lemma}\label{lem:one step lifting lemma}
Let $W\subset V^{\bullet}(b)$ (resp.~$W\subset V^{\circ}(b)$) be a finite-dimensional subspace and let $x_1<x_2\in I_W$. 
Let  $I_1=I_W\cap [x_1,s_1]$ and $I_2=I_W\cap [x_2,s_1]$. 
(If $s_1=+\infty$ then, for each $i$, $I_i=I_W\cap [x_i,s_1)$). 
Then $V^{\bullet}_{I_1}(x_1)\subset V(x_2,x_1)(V^{\bullet}_{I_2}(x_2))$ (resp.~$V^{\circ}_{I_1}(x_1)\subset V(x_2,x_1)(V^{\circ}_{I_2}(x_2))$).
\end{lemma}
\begin{proof} This is a special case of Remarks \ref{rem:V_I(b)}(c) and \ref{rem:V_I(b)}(c'). \end{proof}

\begin{lemma}\label{lem:compatible system of elements}
\item [(a)\ ] For any interval $I$ of the form $[b,c]$ or $[b,c)$ and any element $v\in V^{\bullet}_I(b)$, there is a collection of elements $\{V^{\bullet}_x\in V(x)\}_{x\in I}$ so that $V^{\bullet}_b=v$ and $V(y,x)(V^{\bullet}_y)=V^{\bullet}_x$ for all $b\le x\le y\in I$. 
\item[(a')] For any interval $I$ of the form $(b,c]$ or $(b,c)$ and any element $v\in V^{\circ}_I(b)$, there is a collection of elements $\{v_x\in V(x)\}_{x\in I}$ so that $v_b=v$ and $V(y,x)(v_y)=v_x$ and $V^{\circ}(x,b)(v_x)=v_b$ for all $b< x\le y\in I$. 
\end{lemma}
\begin{proof} 
If $I=[b,c]$ then $V^{\bullet}_I(b)=V(c,b)(V(c))$ and we can choose one element $w\in V^\bullet(c)$ so that $V(c,b)(w)=v$ and let $v_x=V(c,x)(w)$ for all $x\in [b,c]$. 
Similarly, if $I=(b,c]$ then $V^{\circ}_I(b) =V^{\circ}(c,b)(V(c))$ and we make a similar choice. 
 
Otherwise, $I=[b,c)$ or $I=(b,c)$ for some $c>b$. 
In this case, choose an increasing sequence of real numbers $b<x_1<x_2<x_3<\cdots$ in $(b,c)$ converging to $c$.
Let $J_0=I$, for each $l>0$, let $J_l=I\cap [x_l,c)$ and dditionally, let $v_l\in V^{\bullet}_{J_l}(x_l)$ (resp.~$v_l\in V^\circ_{J_l}(x_l)$) be chosen recursively as follows.
\begin{enumerate}
\item Set $v_0=v\in V^{\bullet}_I(b)$, (resp.~$v_0=v\in V^{\circ}_I(b)$).
\item Given $v_k$ in $V^{\bullet}_{J_l}(x_l)$ (resp.~in $V^{\circ}_{J_l}(x_l)$), by Lemma \ref{lem:one step lifting lemma}, there exists $v_{l+1}$ in $V^{\bullet}_{J_{l+1}}(x_{l+1})$ (resp.~$V^{\circ}_{J_{l+1}}(x_{l+1})$) so that $V(x_{k+1},x_k)(v_{k+1})=v_k$ (resp.~$V^{\circ}(x_{l+1},x_l)(v_{l+1})=v_k$).
\end{enumerate}
After this sequence of elements $v_k$ is chosen, the vector $v_x$ for any $x\in I$ is given by $v_x=V(x_l,x)(v_k)$ for any $x_l>x$. This is well defined by condition (2) in the case of $V^{\bullet}(b)$ and by condition (2) combined with the universal property of $V^{\circ}(b)$ in that case.
\end{proof}

\begin{deff}\label{def:one sided projectives}
Let $s_0$ be a sink or $-\infty$ and let $s_1>s_0$ be the next source or $+\infty$.
Let $s_0<a<s_1$.
For $I=[s_0,a]$ or $[s_0,a)$ let $P_I$, also written $P_a=P_{[s_0,a]}$ or $P_{a)}=P_{[s_0,a)}$, denote the representation with support $I$ so that $P_I(x)$ is one-dimensional with generator $v_x$ for all $x\in I$ and $P(y,x)(v_y)=v_x$ for all $x<y\in I$.

For $a=s_1$, define $P_{a)}$ as before.
However, when $a=s_1$, $P_a$ is not defined this way.
If $s_0=-\infty$ then $P_a$ and $P_{a)}$ are instead $P_{(s_0,a]}$ and $P_{(s_0,a)}$, respectively.
\end{deff}

\begin{proposition}\label{prop:one sided at a}
$P_a$ and $P_{a)}$ as in Definition \ref{def:one sided projectives} are projective in $\ReppwfAR$.
\end{proposition}

\begin{proof}
We first assume that $s_0\in\R$.
To show that $P_I$ is projective it suffices to show that any epimorphism $p:E\to P_I$ has a section. Let $W\subset E^\bullet(s_0)$ be the smallest term in the image filtration of $E^\bullet(s_0)$ which maps onto $P_I^
\bullet(s_0)$.

Claim: $I_W$ contains $I$ and thus $W\subset E_I^\bullet(s_0)$.

Proof: For each $x\in I$, there is a $w\in E(x)$ so that $p_x(w)=v_x\in P_I(x)$. But then $p_{s_0}E^\bullet(x,s_0)(w)=P_I^\bullet(x,s_0)(v_x)=v_{s_0}\neq 0$. So, $W\subset E^\bullet(x,s_0)(E(x))$ which implies $x\in I_W$. Since this holds for all $x\in I$ we get that $I\subset I_W$.

By construction of $W$, there is a $w\in W\subset E_I^\bullet(s_0)$ so that $p(w)=v_{s_0}$. By Lemma \ref{lem:compatible system of elements} there are elements $w_x\in E(x)$ for all $x\in I$ so that $E(x,y)(w_x)=w_y$ for all $y\le x\in I$. Then, a section $s:P\to E$ is given by $s(v_x)=w_x$ for all $x\in I$.

If we instead assume $s_0=-\infty$ then above we replace $E^\bullet$ with $E^{\circ}$ and $P^\bullet$ with $P^{\circ}$ where appropriate.
By the universal property of colimits, the map on representations induces a map $E^{\circ}(-\infty)\to P^{\circ}_I(-\infty)$.
Then the rest of the proof holds as stated.
\end{proof}

\begin{lemma}\label{lem:infinite limit is kinfty} 
Suppose $s_0=-\infty$ and $P$ is a pointwise finite-dimensional representation with support $(s_0,a)$, for $a\leq s_1$, or $(s_0,a]$, for $a< s_1$. 
Either $P^{\circ}(s_0)$ is finite-dimensional or $P^{\circ}(s_0)\cong k^\infty$.
\end{lemma}
\begin{proof}
Since $P$ is pointwise finite-dimensional, if the dimension of $P(x)$ is bounded by some $n$ for all $x\in\R$ then $P^{\circ}(s_0)$ is also bounded by $n$.

Now suppose $P^{\circ}(s_0)$ is not finite-dimensional.
For each $i>0$, let $n_i=\dim P^{\circ}_{I_i}$.
Let $e_i\in k^\infty$ denote the unit vector with a 1 in the $i$th coordinate.
For a choice of basis of $P^{\circ}(s_0)$, we note that since each morphism is a monomorphism and the image filtration ($\circ *$) of $P^{\circ}(s_0)$ has a minimal element, we may inductively choose a basis on $P^{\circ}(s_0)$.
We do this by first choosing a basis of $P^\circ_{I_1}$, then completing it to a a basis of $P^\circ_{I_2}$ and so on.
Since each $P^\circ_{I_i}$ is finite-dimensional this is well defined.

Since we have a consistent choice of bases, map the chosen basis of each $P^{\circ}_{I_i}$ to the collection $\{e_i \}\subset k^\infty$ in a consistent way.
Since each $P^{\circ}_{I_i}\cong P(x_i)$ this induces a map $P^{\circ}(s_0)\to k^\infty$.

To see the map is surjective take any element $w$ of $k^\infty$; $w$ has finitely many nonzero coordinates.
Thus it is some linear combination of finitely many $e_j$'s.
Then there is a $P^{\circ}_{I_i}$ whose basis contains enough elements to surject on to the $e_j$'s.
Thus there is an element $v$ in $P_{I_i}$ such that $v\mapsto w$ and so there is an element $\tilde{v}\in P^{\circ}(s_0)$ that maps to $w$.
The map is injective since if $\tilde{v}\neq \tilde{v}'$ in $P^{\circ}(s_0)$ then there is a pair $v\neq v'$ in a $P^{\circ}_{I_i}$ such that $v\mapsto \tilde{v}$ and $v'\mapsto \tilde{v'}$.
We know $v$ and $v'$ map to different elements in $k^\infty$ so $\tilde{v}$ and $\tilde{v}'$ must also.
Therefore, $P^{\circ}(s_0)\cong k^\infty$.
\end{proof}
 
The following theorem will give a characterization of one sided projective objects in $\ReppwfAR$.

 \begin{theorem}\label{thm:characterization of one sided projectives}
 Let $s_0 \leq  a< s_1$ with $s_0$ a sink and $s_1$ the next sourse.
 Let $P$ be a pointwise finite-dimensional representation of $A_{\mathbb R}$ with $\supp P \subset [s_0,a]$. 
 \begin{itemize}
\item[(1)\ ] Then $P$ is projective in $\ReppwfAR$ if and only if all maps $P(x,s_0):P(x)\to P(s_0)$ are injective for all $x\in supp P$.
\item[(2)\ ]   Every projective representation in $\ReppwfAR$ with support in $[s_0,a]$ is a finite direct sum of representations of the forms $P_b$ and $P_{b)}$ for $s_0\le b\le a$.
\item[(2')] Every projective representation in $\ReppwfAR$ with support in $(s_0,a]$ (i.e., $s_0=-\infty$)\ is a possibly infinite direct sum of representations of the forms $P_b$ and $P_{b)}$ for $s_0< b\le a$.
\end{itemize} \end{theorem}

 \begin{proof}
 When $a=s_0$, statements (1) are (2) are trivially true and statement (2') does not apply.
  
 (1) Suppose 
 that there is some $x_0\in 
 [b,a]\subset[s_0,a]$ 
 so that $P(x_0,s_0):P(x_0)\to P(s_0)$ is not injective. Then we will show that $P$ is not projective. Indeed consider the quotient object $Q$ given $Q(x)=P(x)$ for all $x\ge x_0$ and $Q(x)=0$ for all $x<x_0$. We have an epimorphism $\pi: P\to Q$. Let $\widetilde Q$ be the representation given by $\widetilde Q(x)=P(x)$ for $x\ge x_0$ and $\widetilde Q(x)=P(x_0)$ for all $x\le x_0$ with $\widetilde Q(y,x)=Id$, 
 the identity,
  when $x,y\le x_0$ and $\widetilde Q(y,x)=P(y,x_0)$ when $x\le x_0<y$. Let $p:\widetilde Q\to Q$ be the projection map. 
Claim: the quotient map $\pi: P\to Q$ does not lift to $\widetilde Q$, i.e. there is no $\gamma: P\to \widetilde Q$ such that $p\circ\gamma=\pi$.
Proof of claim: Since $\pi_{x_0}=Id: P(x_0)=Q(x_0)$ and $p_{x_0}=Id: \widetilde Q(x_0)\to Q(x_0)$ we would have $\gamma_{x_0}=Id$. But that gives a contradiction to the basic property of maps between representations: $\gamma_{s_0}\circ P(x_0,s_0)=\widetilde Q(x_0,s_0)\circ\gamma_{x_0}=Id$, but $\gamma_{s_0}\circ P(x_0,s_0)$ is not injective by assumption.
 Therefore, $P$ is not projective.

Conversely, suppose that all morphisms $P(x,s_0)$ are monomorphisms. Choose a basis $B$ for $P(s_0)$ compatible with the image filtration. Thus, a subset $B_i$ of $B$ is a basis for each subspace $P_{J_i}(s_0)$ in the image filtration of $P(s_0)$ where 
$J_i=I_{P_{J_i}(s_0)}$ 
are ordered by inclusion: $J_1\subsetneq J_2\subsetneq \cdots\subsetneq J_n$. Then $P_{J_1}(s_0)\supsetneq P_{J_2}(s_0)\supsetneq \cdots \supsetneq P_{J_n}(s_0)$.
 
(2) By Lemma \ref{lem:compatible system of elements}, every $v\in B_i-B_{i+1}$ lifts to a compatible system of elements $v_x\in P(x)$ for all $x\in J_i$. By Lemma \ref{lem:support of image is support}, $P(x,s_0)(P(x))= P_{J_i}(s_0)$ for all $x\in J_i-J_{i-1}$ (where $J_0=\emptyset$). Since $P(x,s_0)$ is a monomorphism, the liftings $v_x\in P(x)$ for all $v\in B_i$ form a basis for $P(x)$ for all $x\in J_i-J_{i-1}$. For each $v\in B_i$, the lifting $v_x$ of $v$ generate a pointwise one-dimensional subrepresentation $Q_v$ of $P$ with support in $J_i$, i.e. $P_v$ has the form $P_{b}$ or $P_{b)}$ depending on whether $J_i=[s_0,b]$ or $[s_0,b)$. The $Q_v$, for $v\in B$, are disjoint and generate all of $P(x)$ for every $x\in [s_0,a]$. Thus $P$ is a direct sum of the $Q_v$ as claimed.

Below is an example of such a decomposition for (2).
\begin{displaymath}\begin{tikzpicture}
\draw[thick] (0,2) -- (0.4,2);
\draw[thick] (0,1.5) -- (1.2,1.5);
\draw[thick] (0,1) -- (3.5,1);
\draw[thick] (0,.5) -- (4,.5);
\draw[thick] (0,0) -- (5,0);
\foreach \x in {0,...,4}
	\filldraw (0, \x / 2) circle[radius=.5mm];
\filldraw[fill=white] (0.4,2) circle[radius=.6mm];
\filldraw (1.2,1.5) circle[radius=.5mm];
\filldraw (3.5,1) circle[radius=.5mm];
\filldraw[fill=white] (4,0.5) circle[radius=.6mm];
\filldraw (5,0) circle[radius=.5mm]; 
\draw (0,2) node[anchor=east] {$[s_0,b_1)=\supp Q_{v_1}$};
\draw (0,1.5) node[anchor=east] {$[s_0,b_2]=\supp Q_{v_2}$};
\draw (0,1) node[anchor=east] {$[s_0,b_3]=\supp Q_{v_3}$};
\draw (0,0.5) node[anchor=east] {$[s_0,b_4)=\supp Q_{v_4}$};
\draw (0,0) node[anchor=east] {$[s_0,b_5] = \supp Q_{v_5}$}; 
\end{tikzpicture}\end{displaymath}

With the exception of choosing a basis, we may apply all of the argument for statement (2) to statement (2').
If $P^{\circ}(s_0)$ is finite-dimensional we get a basis and apply the argument for (2).
By Lemma \ref{lem:infinite limit is kinfty}, if $P^{\circ}(s_0)$ is infinite-dimensional then it is isomorphic to $k^\infty$ and by the proof of the same lemma we have a basis that respects the filtration.
We then apply the argument for (2).
\end{proof}

\begin{example}\label{xmp:no projective cover}
Let $A_\R$ have the straight descending orientation and for each positive integer $n$ let $V_n$ be the following representation:
\begin{align*}
V_n(x) &= \left\{\begin{array}{ll} k & n\leq x \\ 0 & \text{otherwise} \end{array} \right.&
V_n(x,y) &= \left\{\begin{array}{ll} 1_k & n\leq y\leq x \\ 0 & \text{otherwise} \end{array} \right.
\end{align*}
Using Theorem \ref{thm:characterization of one sided projectives} we see that the projective cover of each $V_n$ is the projective indecomposable with support $\R = (-\infty,+\infty)$.

Note that $V=\bigoplus V_n$ is still pointwise finite.
One can check it is isomorphic to the representation $W$ (item (3) in Example \ref{xmp:boundedreps}).
However, the projective cover is infinitely many copies of the indecomposable projective with support $(-\infty,+\infty)$, which is not pointwise finite-dimensional.
Therefore, this rather tame example does not have a projective cover in $\ReppwfAR$. 

However, the dually constructed representation $V'$ (each $V'_n$ has support $(-\infty,n]$) is its own projective cover by Theorem \ref{thm:characterization of one sided projectives} and so does have a projective cover. 
While $V$ and $V'$ exist in $\ReppwfAR$, neither exists in $\RepbAR$.
So, this type of asymmetry does not happen in $\RepbAR$.
\end{example}

\subsection{Sufficient Conditions for Indecomposables}
Here we give sufficient conditions for 
pointwise finite-dimensional representations to be indecomposable. In Section \ref{sec:thetheorem} we will show that these conditions are also necessary. 

\begin{proposition} \label{prop:sufficientProp} Let $V$ be a representation of $A_{\mathbb R}$ such that
\begin{enumerate}
\item $dimV(x)\leq 1$ for all $x\in \mathbb R$,
\item if $V(x)\neq 0\neq V(z)$ and $x\leq y\leq z$  in $\mathbb R$ then $V(y)\neq 0$, and
\item if $V(x)\neq 0\neq V(y)$ and  $x\preceq y$ then $V(y,x)$ is an isomorphism.
\end{enumerate}
Then $V$ is indecomposable.
\end{proposition}
 \begin{proof}
Suppose $V$ is not indecomposable for contradiction. Then $V\cong W_1\oplus W_2$ with $W_1\neq 0\neq W_2$.
Then $supp \,W_1\cap supp \,W_2 = \emptyset$ and $supp \,W_1\cup supp \,W_2 = supp\,V.$ 
Since $W_1,W_2\neq 0$  there exist $x_1\in supp\, W_1$ and $x_2\in supp\,W_2$. 
By symmetry we may assume $x_1<x_2$.

\smallskip
Claim 1: There are only finitely many elements of $S$ in the open interval $(x_1,x_2)$.\\
Pf: This follows from the fact that $(x_1,x_2)$ is a bounded interval, i.e., $[x_1,x_2]$ is compact. If $S\cap [x_1,x_2]$ were infinite, it would contain a converging sequence (any infinite subset of a compact set contains a converging sequence). By definition, $S$ does not contain a converging sequence. So, $S\cap [x_1,x_2]$ is finite. A fortiori, $S\cap (x_1,x_2)$ is finite.

\smallskip
Claim 2: There exist $x_1\in supp\, W_1$ and $x_2\in supp\,W_2$ such that $S\cap (x_1,x_2)=\emptyset.$\\
Pf: Let $n=\#\{S\cap (x_1,x_2)\}$. If $n\ge1$ we will find another pair $x_1'<x_2'$ in the respective supports of $W_1,W_2$ so that $\#\{S\cap (x_1',x_2')\}<\#\{S\cap (x_1,x_2)\}=n$. This will imply that $n=0$.

To find this second pair $x_1',x_2'$ choose any element $s_k$ in $S\cap (x_1,x_2)$ which is nonempty by assumption that $n\ge1$. Then $s_k$ is in the support of $W_1$ or $W_2$. In the first case, $x_1'=s_k,x_2'=x_2$ gives the desired pair. Indeed, in this case, $S\cap (x_1',x_2')\subset S\cap (x_1,x_2)$ since $s_k\in S\cap (x_1,x_2)$ but $s_k\notin S\cap (x_1',x_2')$. Also, $x_1'=s_k$ is in the support of $W_1$ by assumption and $x_2'=x_2$ is in the support of $W_2$. The second case is similar. In both cases, the value of $n$ can be reduced if it is positive. So, the minimal value of $n$ is 0.

\smallskip
By Claim 2 we may assume there are no elements of $S$ between $x_1$ and $x_2$, i.e. the $\prec$ orientation of $\mathbb R$ is constant in the closed interval $[x_1,x_2]$ and either $V(x_2,x_1)$ or $V(x_1,x_2)$ is an isomorphism. In the first case, we consider the projection $V\to W_1$ and in the second case we consider the other projection $f:V\to W_2$. By symmetry, we may take the first case, i.e. $V(x_2,x_1)$ is an isomorphism. Then we have the following commuting diagram:
\[
\xymatrix{
V(x_1)\ar[d]^{f_{x_1}}_{\cong}&
	V(x_2)\neq 0\ar[d]^{f_{x_2}}\ar[l]_{V(x_2,x_1)}^{\cong}\\
W_1(x_1) & 
	W_1(x_2)=0 \ar[l]
	}
\]
Since $x_1\in supp\,W_1$, it follows that $f_{x_1}:V(x_1) \to W_1(x_1)$ is an isomorphism. 
But $W_1(x_2)=0$ since $x_2 \in supp\, W_2$ and $x_2 \notin supp\, W_1$. The commutativity of the diagram then gives a contradiction. Thus $V$ is indecomposable.
\end{proof}

\begin{definition}\label{def: MI}
For any interval $I$ in $\R$ let $M_I$ be the representation of $A_\R$ given as follows.
\begin{align*}
M_I(x) &= \left\{\begin{array}{ll} k & x\in I
 \\ 0 & \text{otherwise} \end{array}
 \right. 
 & 
M_I(x,y) &= \left\{\begin{array}{ll} 1_k & y\preceq x \text{ and }x,y\in I \\ 0 & \text{otherwise} \end{array}\right.
\end{align*} 
The conditions of Proposition \ref{prop:sufficientProp} are satisfied immediately. So, $M_I$ is indecomposable.
If a representation $V\cong M_I$ we call $V$ an \underline{interval indecomposable} or \underline{interval indecomposable representation}.
\end{definition}

\begin{cor} Let $V$ be an indecomposable representation which is pointwise one-dimensional (satisfies the conditions of Proposition \ref{prop:sufficientProp}). Let $J\subseteq supp V$ be a connected subset and let $V_J$ be the restriction of $V$ to $J$, i.e. $V_J(x)=V(x) =k$ for all $x\in J$ and $V_J(x)=0$ for all $x\notin J$, and $V_J(y,x)=V(y,x)$ for all $x,y\in J$. Then $V_J$ is an indecomposable representation.
\end{cor}
\begin{proof} We will show that conditions (1), (2) and (3) of Proposition \ref{prop:sufficientProp} are satisfied by the representation $V_J$.

(1) By definition of $V_J$ it follows that $\dim_kV_J(x)\leq 1$.

(2) This follows since $J$ is connected subset of $supp\,V$. 

(3) Suppose there is $x \preceq y$ with $x,y\in J$ such that $V_J(y,x)=V(y,x)$ not an isomprhism. Since $\dim_kV_J(x)\leq 1$, this is equivalent to $V_J(y,x)=0$.\\
Let $I=\{t\ |\ x\prec t\preceq y \text{ such that } V(y,t)\neq 0\}$. Then $x\in (supp V)\backslash I=J_1\cup J_2$ where $J_1\cap J_2=\emptyset$ and we may assume $x\in J_1$. Then $V_{J_1}$ is a subrepresentation of $V$ but is also a quotient of $V$ since the map $\pi: V\to V_{J_1}$ defined as $\pi_x=I\!d_{V_J(x)}$ for $x\in J_1$ and $\pi_x=0$ for $x\notin J_1$ is a representation homomorphism using the fact that $V(t_2,t_1)=0$ for all $t_2\in (supp V)\backslash J_1$ and all $t_1\in J_1$. Actually $\pi$ is a splitting for the inclusion $V_{J_1}\to V$,  contradicting the assumption that $V$ is indecomposable. Therefore (3) holds for $V_J$.

So by Proposition \ref{prop:sufficientProp} it follows that $V_J$ is indecomposable.
\end{proof}

\subsection{Filtrations}\label{sec:filtrations}
In this section will provide some lemmas necessary for Section \ref{sec:thetheorem}.
In both this section and in Section \ref{sec:thetheorem} we will be using notation $\Hom(\_\,,\_)$ for $\Hom_{\ReppwfAR}(\_\,,\_)$ and $\End(\_)$ for $\End_{\ReppwfAR}(\_)$  where $\ReppwfAR$ is the full subcategory of $\RepAR$  whose objects are all pointwise finite representations of $A_{\mathbb R}$.

\begin{lemma}\label{lem:V is schurian}
Let $V$ be an indecomposable pointwise one-dimensional representation. 
Then the endomorphism ring of $V$ is the field $k$.
\end{lemma}

\begin{proof}
Let $x_0\in \supp V$.
By definition, $V(x_0)\cong k$.
Choose a morphism $f(x_0):V(x_0)\to V(x_0)$.
\underline{Claim}: if $f(x_0)\neq 0$ this determines an isomorphism $V\stackrel{\cong}{\to} V$.

Since $V(x_0)\cong k$, $f(x_0)$ is an isomorphism.
If $y\in \supp V$ such that $y\preceq x_0$ then $V(x_0,y)$ is an isomorphism.
So, for all $y\preceq x_0$ in $\supp V$, define \begin{displaymath} f(y):=V(x_0,y)\circ f(x_0) \circ (V(x_0,y))^{-1}.\end{displaymath}
Dually, for all $y\in\supp V$ such that $x_0\preceq y$ define \begin{displaymath} f(y):=(V(y,x_0))^{-1}\circ f(x_0) \circ V(y,x_0).\end{displaymath}
If there are no sinks and sources in $\supp V$, except possibly the endpoints, we have an induced morphism $V\to V$ such that $f(x)$ is an isomorphism for all $x\in \R$ (by setting $f(x)=0$ when $x\notin \supp V$).
By Proposition \ref{prop:isomorphisms} $f$ is an isomorphism.
Now, suppose there is a sink or source in the interior of $\supp V$.

Let $s_n$ be a source such that $x_0\preceq s_n$.
By the paragraph above we already have $f(s_n)$.
For each $y\preceq s_n$ for which we do not yet have an $f(y)$ we can use the technique above and define it without making choices.
By a dual argument if $s_n\preceq x_0$ we can define $f(y)$ for all $y$ such that $s_n\preceq y$.
Note that between any real number $x$ and $x_0$ there are only finitely many sinks and sources between $x$ and $x_0$ in the total oder of $\R$.
By repeated use of this technique, we have an induced isomorphism $f(x):V(x)\stackrel{\cong}{\to}V(x)$ for all $x\in \R$.
Thus, we have an induced isomorphism $f:V\stackrel{\cong}{\to} V$.
If $g:V\to V$ is a nonzero morphism then $g(x)$ is nonzero as before is an isomorphism that determines the rest of $g$.
Then $g(x)$ and $f(x)$ are multiplication by nonzero scalars and there exists $t\in k$ such that $t g(x)=f(x)$.
Therefore, $\End(V)\cong k$.
%
%
%
%
\end{proof}
 
 \begin{theorem} \label{thm:iso-indecomps}Let $V$ and $V'$ be two indecomposable pointwise one-dimensional representations of $A_{\mathbb R}$.
Then $\supp V=\supp V'$ if and only if $V\cong V'$.
 \end{theorem}
 \begin{proof}
We first assume $\supp V=\supp V'$.
Let $x_0\in \supp V=\supp V'$.
By definition, $V(x_0)\cong k \cong V'(x_0)$.
Choose an isomorphism $f(x_0):V(x_0)\stackrel{\cong}{\to} V'(x_0)$ and apply the argument from Lemma \ref{lem:V is schurian}.
%
%
The reverse direction is a special case of Proposition \ref{prop:isomorphic implies same support}.
\end{proof}
 
\begin{deff}
 Let $X_1\subset X_2\subset \cdots\subset X_n$ be a filtration of a vector space $X=X_n$. A basis $B$ for $X$ is said to \underline{respect the filtration} if $B\cap X_j$ is a basis for $X_j$ for each $j$. A direct sum decomposition $X=\bigoplus Y_i$ of $X$ is said to \underline{respect the filtration} if each $X_j$ is a direct sum of some of the $Y_i$.
 \end{deff}

 \begin{lemma}\label{lem:Lemma X}
 For any $b\in \mathbb R$, let $\mathcal V_b$ be the full subcategory of $\ReppwfAR$ whose objects are 
interval indecompsables $V$ with $b\in supp(V)\subset[b,\infty)$. Let $\mathcal W_b:= add\, \mathcal V_b$. Then:
 \begin{enumerate}
 \item The restriction map, $res: \mathcal V_b\to \Rep_k(\{b\})$ given by $res(V)=V(b)$ and $res(f)= f(b)$ defines a monomorphism
 $\Hom_{\mathcal V_b}(V,V')\to \Hom_k(V(b),V'(b))$ 
 for all $V,V'\in \mathcal V_b$, i.e. restriction to $b$ is a faithful functor on $\mathcal V_b$.
\item The restriction map, $res: \mathcal W_b\to \Rep_k(\{b\})$ is also a faithful functor on $\mathcal W_b$.
\item   There is a unique total ordering on the set of isomorphism classes of objects of $\mathcal V_b$ so that: 
(a) $\Hom_{\mathcal V_b}(V',V)=0$ and $\dim_k\Hom_{\mathcal V_b}(V,V')=1$ whenever $V>V'$ and \\
(b) composition of nonzero maps $V\to V'\to V''$ is always nonzero. 
\item Any $W\in \mathcal W_b$ has a unique filtration $0=W_0\subset W_1\subset \cdots \subset W_m$ so that each $W_k/W_{k-1}$ lies in $add\,V_k$ where $V_1<V_2<\cdots<V_m$. 
Evaluating at vertex $b$ we get a filtration $W_1(b)\subset W_2(b)\subset\cdots\subset W_m(b)=W(b)$ which we call the \underline{filtration of} $W(b)$ \underline{induced by the filtration of} $W$.
\item For any $W\in \mathcal W_b$, any direct sum decomposition $W(b)=\bigoplus X_{l}$ of $W(b)$ into one-dimensional subspaces which respects the filtration of $W(b)$ induced from the 
filtration of $W$ extends to a direct sum decomposition of $W$, i.e. $W=\bigoplus Y_{l}$ so that $Y_{l}(b)=X_{l}$ for all $l$.
\end{enumerate}
 \end{lemma}
 
 \begin{proof} (1) Given $V,V'$ in $\mathcal V_b$, the support of one of them contains the support of the other. 
 Let $J=supp\,V\cap supp\,V'\subset [b,\infty)$. 
 Then either $J=supp\,V$ or $J=supp\,V'$. Suppose $J=supp\,V$. Since $V$ is indecomposable $J$  is connected and $J\subseteq supp V'$.
 Any morphism $f:V\to V'$ induces a morphism $f_J:V_J\to V_J'$ by restricting to $J$. By Theorem \ref{thm:iso-indecomps}, $V_J\cong V_J'$ is either $V$ or $V'$.
Then we have, by Lemma \ref{lem:V is schurian}, that $f_J$ is a scalar times a fixed isomorphism $V_J\cong V_J'$. In particular $f=0$ if and only if $f$ is zero at $b$. So, evaluation at $b$ is faithful. (2) follows immediately from (1).
 
 (3) Given $V,V'$ in $\mathcal V_b$, suppose by symmetry that the support of $V$ is properly contained in the support of $V'$. Then, there is some $m>b$ so that the support of $V$ is contained in $[b,m]$. There are only finitely many elements of $S$ inside this compact set. Without loss of generality we may assume that $b\in S$. 
 Let $l$ be maximal so that $s_{l}$ is in the support of $V$. If $s_{l}$ is a sink, then $V$ is a sub-representation of $V'$. If $s_{l}$ is a source, then $V$ is a quotient representation of $V'$. In the first case, $\Hom(V,V')=k$ and $\Hom(V',V)=0$. In the second case, $Hom(V,V')=0$ and $Hom(V',V)=k$.
 
 If there are nonzero morphisms $V\to V'\to V''$ then, by (1), evaluation at $b$ gives isomorphisms $V(b)\cong V'(b)\cong V''(b)$. So, the relation of having a nonzero morphism $V\to V'$ is transitive, reflexive, antisymmetric and any two elements are related. So, this is a total ordering.
 
 (4) Given $W\in \mathcal W_b$ we have by definition a direct sum decomposition $W=\bigoplus_{1\le i\le m} (V_i)^{n_i}$ where we order the summands according to the total order given in (3). So, there exists a filtration $0=W_0\subset W_1\subset W_2\subset \cdots\subset W_m=W$ so that $W_i/W_{i-1}=n_iV_i$. Since $Hom(V_i,V_j)=0$ for $i<j$, the sub-representation $W_i$ is uniquely characterized as the trace of $V_1\oplus V_2\oplus \cdots\oplus V_i$ in $W$. So, the HN-filtration is unique.
 
 (5) Let $n=\dim W(b)$ and let $G$ be the subgroup of $GL(n,k)$ which preserves the filtration $W_1(b)\subset W_2(b)\subset\cdots\subset W_m(b)$. This is a block upper triangular matrix group which acts transitively on the set of all bases which respect this filtration of $W(b)$. Since $Hom(V_i,V_j)=K$ for $i\le j$ and $Hom(V_i,V_j)=0$ for $i>j$, we have by (3) that the restriction map $Aut(W)\to Aut(W(b))=G$ is an isomorphism. Therefore, $Aut(W)$ acts transitively on the set of all bases for the vector space $W(b)$ which respect the given filtration of $W(b)$. 
 
Recall we are given a direct sum decomposition $W(b)=\bigoplus X_l$ into one-dimensional subspaces that respects the induced filtration and $W=\bigoplus Y_l$ a direct sum decomposition of $W$ into pointwise one-dimensional indecomposable representations.
One such basis is given by choosing a generator $x_{l}\in X_{l}$ for each summand $X_{l}$ of $W(b)=\bigoplus X_{l}$. A second basis is given by choosing a generator $y_{l}\in Y_{l}(b)$ where $W=\bigoplus (V_i)^{n_i}=\bigoplus Y_{l}$, where each $Y_l$ is equal to some $V_i$, is the given decomposition of $W$ into indecomposable representations which are one-dimensional at $b$. Take $\varphi\in G$ which takes $(y_{l})$ to $(x_{l})$. Then $W=\bigoplus\varphi(Y_{l})$ is the required decomposition of $W$ extending the chosen decomposition of $W(b)$.
 \end{proof}

 \begin{lemma}\label{lem:Y}
 Given any two finite filtrations of a finite-dimensional vector space $X$, there exists a direct sum decomposition of $X$ into one-dimensional subspaces which respects both filtrations.
 \end{lemma}
 
 \begin{proof}
 Given any two filtrations $V_1\subset V_2\subset \cdots\subset V_n=X$ and $W_1\subset W_2\subset\cdots \subset W_m=X$ of $X$ we have the following representation of a quiver of type $A_{n+m-1}$:
 \[
 \xymatrixrowsep{10pt}\xymatrixcolsep{20pt}
\xymatrix{
M:& V_1 \ar[r]^\subset& 
	V_2 \ar[r]^\subset& \cdots\ar[r]^\subset & V_{n-1}\ar[r]^\subset & X & W_{m-1}\ar[l]_{\supset}& \cdots\ar[l]_\supset & W_2\ar[l]_\supset & W_1\ar[l]_\supset.
	}
 \]
We have a direct sum decomposition $M=\bigoplus M_i$ where each $M_i$ is one-dimensional at the middle vertex. This gives a direct sum decomposition of $X$ into one-dimensional subspaces. Then it suffices to prove the following.

Claim: This decomposition $X=M(n)=\bigoplus M_i(n)$ respects both filtrations.

 Proof: Since the maps in the representation $M$ are all monomorphisms, the same holds for each indecomposable component $M_i$. So, each component is nonzero at vertex $n$ (where $M(n)=X$). For any $1\le j< n$, consider the set $I_j$ of all indices $i$ so that $M_i(j)\neq 0$. Then the sum of all $M_i(n)$ for all $i\in I_j$ is equal to $V_j$. Thus $\bigoplus M_i(n)$ respects the first filtration $V_i$ of $X$. Similarly, $\bigoplus M_i(n)$ respects the second filtration $W_j$. So, it respects both filtrations. This proves the lemma.
\end{proof}

\subsection{Necessary Conditions and Decomposition Theorem}\label{sec:thetheorem}

In this section we prove that the sufficient conditions in Proposition \ref{prop:sufficientProp} are also necessary conditions.
We then work up to Lemmas \ref{lem:decomposition extension} and \ref{lem:split subrepresentation}.
Lemma \ref{lem:decomposition extension} shows that the decomposition of certain subrepresentations may be extended to infinity.
Lemma \ref{lem:split subrepresentation} then states that these subrepresentations are indeed summands.
Because we may have infinitely many sinks and sources in our continuous quiver, these lemmas are an essential component of the proof of Theorem \ref{thm:GeneralizedBarCode}, the decomposition theorem.
We save our discussion relating our proof of Theorem \ref{thm:GeneralizedBarCode} to the decomposition theorems in \cite{Botnan,BotnanCrawley-Boevey,Crawley-Boevey2015} for Section \ref{sec:relation to BC-B}.

\begin{deff}\label{def:dualdef}
Choose $A_\R$, a continuous quiver of type $A$.
The \underline{opposite quiver} of $A_\R$, denoted $A_\R^{\text{op}}$, is the continuous quiver of type $A$ where $x\preceq y$ in $A_\R^{\text{op}}$ if and only if $y\preceq x$ in $A_\R$.

Let $V$ be a pointwise finite representation of $A_\R$.
The \underline{dual representation} of $V$, denoted $DV$, is the pointwise finite representation of $A_\R^{\text{op}}$ given by
\begin{align*}
DV(x) :& = D(V(x)) &
DV(y,x) :& = D(V(x,y))
\end{align*}
\end{deff}

\begin{remark}
In Definition \ref{def:dualdef}, note that since $V$ is pointwise finite we have $DDV\cong V$.
\end{remark}

 \begin{lemma}\label{lem:lem A}
 Let $V$ be any object of $\ReppwfAR$. Then the restriction $V_J$ of $V$ to any closed interval $J=[a,b]$ where $s_n\leq a<b\leq s_{n+1}$ for some $n\in\mathbb Z$, decomposes as $V_J=A\oplus B$ where $A$ has support in the open interval $(a,b)$ and $B$ is a finite direct sum of indecomposable one-dimensional representations which are nonzero at either $a$, $b$ or both. 
  \end{lemma}
 
 \begin{proof}
 Without loss of generality we may assume that $s_n$ is a sink and $s_{n+1}$ is a source. Let $K$ be the subrepresentation of $V_J$ given by $K(x)=\ker(V(x,n):V(x)\to V(s_n))$ for all $x\in J=[s_n,s_{n+1}]$. 
 
 Then $P=V_J/K$ is a projective representation of $J$ since all morphisms $P(x)\to P(s_n)$ are monomorphisms by construction of $K$. So $V_J\cong K\oplus P$. It is straightforward to decompose $P$ as a direct sum of finitely many pointwise one-dimensional representations, each nonzero at $n$.
 
 It remains to show that $K$ is a direct sum of a representation with support on the open interval $(s_n,s_{n+1})$ and a finite number of pointwise one-dimensional representations all nonzero at $s_{n+1}$. This is accomplished using the dual representation $DK$. Since $DK$ is a representation of the opposite quiver, the interval $J$ with $n$ as source and $s_{n+1}$ as sink, using exactly the same argument as above we see that $DK=A\oplus B$ where $A_{n+1}=0$ and $B$ is a projective representation of $J^{op}$. Thus $K\cong DA\oplus DB$ where $DA$ has support in the open interval $(s_n,s_{n+1})$ and $DB$ is a finite direct sum of one-dimensional representations which are all nonzero at $s_{n+1}$.
 \end{proof}
 
 \begin{lemma}\label{lem:lem B} If $V$ is an indecomposable object of $\ReppwfAR$ with support in an interval $[s_n,s_{n+1}]$ for some $n\in \mathbb Z$, then $V$ is pointwise one-dimensional.
 \end{lemma}
 
 \begin{proof}
 The support of $V$ must be an interval $J\subseteq[s_n,s_{n+1}]$. If $J$ contains either of its endpoints then the previous lemma applies. It remains to consider the case when $J=(a,b)$ is open. Let $c\in (a,b)$. Then applying the previous lemma to the intervals $[a,c]$ and $[c,b]$ we decompose $V_{[a,c]}$ and $V_{[c,b]}$ into a direct sum of finitely many pointwise one-dimensional representations each of which is nonzero at $c$. The other components of $V_{[a,c]}$ and $V_{[c,b]}$ given by the lemma must be zero since they would be components of $V$. This is equivalent to a representation of a finite quiver of type $A_m$ with straight orientation. So, we can choose the decompositions of $V_{[a,c]}$ and $V_{[c,b]}$ so that they give the same decomposition of $V_c$. This decomposes $V=V_{[a.b]}$ into a direct sum of pointwise one-dimensional representations. Since $V$ is indecomposable there is only one component. 
 \end{proof}

 \begin{lemma}\label{lem:lem C} Let $V$ be an indecomposable object of $\ReppwfAR$. For any two integers $n<m$, the restriction of $V$ to the closed interval $[s_n,s_m]$ is a direct sum of finitely many indecomposable pointwise one-dimensional representations.
 \end{lemma}
 
 \begin{proof}
 The proof is by induction on $m-n$. Suppose first that $m-n=1$ and let $J=[s_n,s_m]=[s_n,s_{n+1}]$. If $supp V\subseteq J$ then $V$ is pointwise one-dimensional by Lemma \ref{lem:lem B}. If $supp V\nsubseteq J$ then, by Lemma \ref{lem:lem A}, the restriction of $V$ to $J$ is a direct sum of pointwise one-dimensional objects plus a summand $A$ with support on $(s_n,s_{n+1})$. But such a summand would also be a summand of $V$ by Lemma \ref{lem:W(b)=0 implies W is summand}. Therefore, $A=0$ and the Lemma holds for $m=n+1$.
 
 Now suppose $m\ge n+2$ and take any integer $k$ so that $n<l<m$. By induction on $m-n$, $V_{[s_n,s_{l}]}$ and $V_{[s_{l},s_m]}$ decompose into pointwise one-dimensional components. By Lemma \ref{lem:Lemma X}, this gives two filtrations of $V_k$. By Lemma \ref{lem:Y}, there is a direct sum decomposition of $V_n$ compatible with both filtrations. This extends to compatible direct sum decompositions of $V_{[s_n,s_{l}]}$ and $V_{[s_{l},s_m]}$ which paste together to give a decomposition of $V_{[s_n,s_m]}$ into one-dimensional representations.
 \end{proof}
 
  \begin{lemma}\label{lem:W(b)=0 implies W is summand}
 Let $V$ be a representation in $\ReppwfAR$ and let $V_{(-\infty,b]}$ be the restriction of $V$ to the interval $(-\infty,b]$. Then any summand $W$ of $V_{(-\infty,b]}$ which is zero at $b$ is a summand of $V$.
 \end{lemma}
 
 \begin{proof}
 Let $\pi:V_{(-\infty,b]}\to V_{(-\infty,b]}$ be the projection to $W$. Then $\pi_b:V(b)\to V(b)$ is zero. So, $\pi$ and the zero morphism on $V_{[b,\infty)}$ agree on the overlap of their domains. So, their union is an endomorphism of $V$. This endomorphism is evidently the projection to $W$ showing that $W$ is a summand of $V$.
 \end{proof}
 
\begin{construction}\label{con:V0infinity}
Let $A_\R$ be a continuous quiver of type $A$ whose sinks and sources are unbounded above.
I.e., for each sink or source $s_n$ there is an $s_{n+1}$.
Let $V$ be a pointwise finite-dimensional representation of $A_\R$ such that, for all $n\in\Z$, the restriction $V_{[s_n,s_{n+1}]}$ contains no direct summands whose support is contained entirely in $(s_n,s_{n+1})$ (i.e., $A=0$ in the $A\oplus B$ decomposition in Lemma \ref{lem:lem A}).

Consider the restriction $V_{[s_{l-1},s_l]}$.
By assumption $V_{[s_{l-1},s_l]}$ is a finite direct sum of indecomposables, all of whose support includes $s_l$ or $s_{l-1}$.
Let $V^l_{0,0}$ be the direct sum of all those summands that include \emph{only} $s_l$, \emph{not} $s_{l-1}$.
Now consider $V_{[s_{l-1},s_{l+1}]}$.
By assumption $V_{[s_{l-1},s_{l+1}]}$ is a finite direct sum of indecomposables, each of whose support contains $s_{l-1}$, $s_l$ or $s_{l+1}$.
Let $V^l_{0,1}$ be the direct sum of all such indecomposables whose support contains \emph{both} $s_l$ \emph{and} $s_{l+1}$, but \emph{not} $s_{l-1}$.
Let $V^l_{1,1}$ be the direct sum of such indecomposables whose support contains $s_l$ \emph{but not} $s_{l+1}$ or $s_{l-1}$.
We ignore those indecomposables whose support does not contain $s_l$.

We can continue this process for all $n\geq 0$.
For each $n\geq 0$ and $0\leq i \leq n$ we define $V^l_{i,n}$ in the following way.
It is the direct sum of those summands of $V_{[s_{l-1},s_{l+n}]}$ whose support contains exactly sinks and sources $s_{l+j}$ for $0\leq j \leq n-i$.
Note that this never includes $s_{l-1}$.
In particular, $V^l_{0,n}$ is the direct sum of those interval indecomposable summands of $V_{[s_{l-1},s_{l+n}]}$ whose support contains $s_l$ and $s_{l+n}$.
We have three examples below, two from the previous paragraph and also the summands we consider from $V_{[s_{l-1},s_{l+2}]}$.
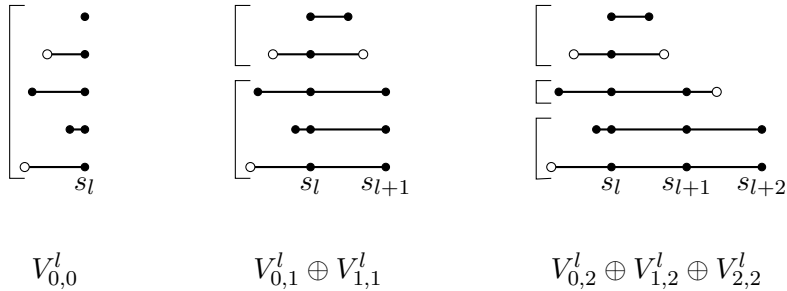
\begin{displaymath}\begin{tikzpicture}
\foreach \y in {0,...,4}
	\filldraw[fill=black] (0,\y/2) circle[radius=.5mm];
	
\foreach \y in {0,1,2}
{
	\draw[thick] (3, \y/2) -- (4,\y/2);
	\foreach \x in {3,4}
		\filldraw[fill=black] (\x,\y/2) circle[radius=.5mm];
}

\foreach \y in {0,1}
{
	\draw[thick] (7,\y/2) -- (9,\y/2);
	\foreach \x in {7,8,9}
		\filldraw[fill=black] (\x,\y/2) circle[radius=.5mm];
}

\draw[thick] (7,1) -- (8.4,1);
\filldraw[fill=white] (8.4,1) circle[radius=.6mm];
\filldraw[fill=black] (7,1) circle[radius=.5mm];
\filldraw[fill=black] (8,1) circle[radius=.5mm];

\foreach \x in {3,7}
{
	\foreach \y in {3,4}
		\filldraw[fill=black] (\x,\y / 2) circle[radius=.5mm];
	\draw[thick] (\x, 1.5) -- (\x+0.7, 1.5);
	\draw[thick] (\x, 2) -- (\x+0.5, 2);
	\filldraw[fill=white] (\x+0.7,1.5) circle [radius=.6mm];
	\filldraw[fill=black] (\x+0.5,2) circle [radius=.5mm];
}

\draw (-0.4,-1) node[anchor=north] {$V^l_{0,0}$};
\draw (3.1,-1) node[anchor=north] {$V^l_{0,1}\oplus V^l_{1,1}$};
\draw (7.6,-1) node[anchor=north] {$V^l_{0,2}\oplus V^l_{1,2}\oplus V^l_{2,2}$};

\foreach \x in {0, 3 ,7}
	\draw (\x,0) node[anchor=north] {$s_l$};
	
\foreach \x in {4,8}
	\draw (\x,0) node[anchor=north] {$s_{l+1}$};
	
\draw (9,0) node[anchor=north] {$s_{l+2}$};

\draw(-.8,-.15) -- (-1,-.15) -- (-1,2.15) -- (-.8,2.15);

\draw(2.2, -.15) -- (2, -.15) -- (2, 1.15) -- (2.2, 1.15);
\draw(2.2, 1.35) -- (2, 1.35) -- (2,2.15) -- (2.2,2.15);

\draw(6.2, -.15) -- (6, -.16) -- (6, .65) -- (6.2, .65);
\draw(6.2, .85) -- (6, .85) -- (6, 1.15) -- (6.2, 1.15);
\draw(6.2, 1.35) -- (6, 1.35) -- (6, 2.15) -- (6.2, 2.15);

\foreach \x in {0, 3, 7}
{
	\draw[thick] (\x - 0.8, 0) -- (\x,0);
	\filldraw[fill=white] (\x-.8,0) circle [radius=.6mm];
	\draw[thick] (\x - 0.7, 1) -- (\x,1);
	\filldraw[fill=black] (\x-.7,1) circle [radius=.5mm];
	\draw[thick] (\x - 0.5, 1.5) -- (\x,1.5);
	\filldraw[fill=white] (\x-.5,1.5) circle [radius=.6mm];
	\draw[thick] (\x - 0.2, .5) -- (\x,.5);
	\filldraw[fill=black] (\x-.2,.5) circle [radius=.5mm];
}
\end{tikzpicture}\end{displaymath}
We note that if $1\leq i \leq n$ then $V^l_{i,n} = V^l_{i+1,n+1}$. 
Note also that the constructions can be made on $(-\infty, s_{l+1}]$ instead and those representations are denoted $V^{i,n}_l$. 
\hfill $\diamond$
\end{construction}

\begin{proposition}\label{prop:V1n is split}
Let $A_\R$ and $V$ be as in Construction \ref{con:V0infinity} for some $s_l$.
Then, for all $n\geq 1$ and $1\leq i \leq n$, $V_{i,n}^l$ and $V^{i,n}_l$ are split subrepresentations of $V$.
\end{proposition}
\begin{proof}
We know the representation $V_{1,n}$ is a split subrepresentation of $V_{[s_{l-1},s_{l+n+1}]}$ as a consequence of Lemma \ref{lem:lem C}.
We know that $V_{1,n}(s_{l-1})=0$ and $V_{1,n}(s_{l+n+1})=0$.
By two uses of Lemma \ref{lem:W(b)=0 implies W is summand} we see that $V_{1,n}$ is a split subrepresntation of $V$.
Finally, recall that $V^l_{i,n} = V^l_{i+1,n+1}$ when $i\geq 1$.
By a similar argument $V^{i,n}_l$ is a split subrepresntation of $V$.
\end{proof}

\begin{lemma}\label{lem:decomposition extension} 
Let $l\in \Z$ and $V$ be a representation with support contained in $[s_{l-1},+\infty)$. 
Assume that for all $n\geq 0$, any indecomposable summand of $V_{[s_{l-1},s_{l+n}]}$ has support at $s_{l+n}$.
Then a decomposition of $V_{[s_{l-1},s_{l+n}]}$ into interval indecomposables extends to a decomposition of $V_{[s_{l-1},s_{l+n+1}]}$.
\end{lemma}
\begin{proof}
Suppose $V_{[s_{l-1},s_{l+n}]}\cong \bigoplus_i M_{I_i}$ is a decomposition.
Then each $I_i$ includes $s_{l+n}$.

If $s_{l+n}$ is a sink then $V(s_{l+n},s_{l+n+1})$ is a monomorphism. 
Any interval indecomposable summands of $V_{[s_{l-1},s_{l+n+1}]}$ that do not have support at $s_{l+n}$ are projective.
In particular, they are split subrepresentations of the same restriction (combine Lemmas \ref{lem:lem A} and \ref{lem:W(b)=0 implies W is summand}).
Let $U_{n+1}$ be the quotient of $V_{[s_{l-1},s_{l+n+1}]}$ by the these projective interval indecomposables.
Since $(U_{n+1})_{[s_{l-1},s_{l+n}]} = V_{[s_{l-1},s_{l+n}]}$ and $U_{n+1}(s_{l+n},s_{l+n+1})$ is an isomorphism we can extend the decomposition to $U_{n+1}$.
Since $U_{n+1}$ is a decomposable summand of $V_{[s_{l-1},s_{l+n+1}]}$ and the other summand is decomposable by Theorem \ref{thm:characterization of one sided projectives} we have extended our decomposition.

If $s_{l+n}$ is a source then $V(s_{l+n},s_{l+n+1})$ is an epimorphism. 
Any interval indecomposable summands of $V_{[s_{l-1},s_{l+n+1}]}$ that do not have support at $s_{l+n}$ are injective.
They are split subrepresentations as before.
We can now apply the same argument in the previous paragraph and extend the decomposition.
\end{proof}

The assumptions in the following lemma are justified by Proposition \ref{prop:V1n is split}.

\begin{lemma}\label{lem:split subrepresentation}
Let $A_\R$ and $V$ be as in Construction \ref{con:V0infinity}.
For all $l\in \Z$, $n\geq 1$, and $1\leq i \leq n$, assume $V_{i,n}^l=0=V^{i,n}_l$.
Then $V$ contains a summand as in Lemma \ref{lem:decomposition extension} but whose support does not contain $s_{l-1}$.
\end{lemma}
\begin{proof}
For each $n\geq 1$ and a decomposition of $V_{[s_{l-1},s_{l+n}]}$ let $K_n$ be the sum of interval summands whose support is nonzero at $s_{l-1}$.
By assumption, if $s_{l-1} \leq x \leq y \leq s_{l+n}$ then $\dim K_n(x) \geq \dim K_n(y)$.
Note that $\dim K_n(x) = \dim K_{n+1}(x)$ on $[s_{l-1},s_{l+n}]$, though the decomposition of $K_n$ is not assumed to extend exactly.

Therefore we have a function $[s_{l-1},+\infty)\to \N$ that is weakly decreasing and whose initial value is finite.
Therefore, the function must stabilize to some particular value.
Let $m$ be sufficiently large that $\dim K_n(s_{l+n}) = \dim K_{n+1}(s_{l+n+1})$ for all $n\geq m$.
Then, by assumption, every map $V(s_{l+n},s_{l+n+1})$ for $n\geq m$ is mono or epi. 
So we can use the same technique in Lemma \ref{lem:decomposition extension} to extend a decomposition of $V_{[s_{l-1},s_{l+m+1}]}$ to all of $V_{[s_{l-1},+\infty)}$. 

Then any summands of $V_{[s_{l-1},+\infty)}$ with bounded support that is nonzero at $s_{l-1}$ are split subrepresentations of $V_{[s_{l-1},+\infty)}$ (Lemma \ref{lem:W(b)=0 implies W is summand}). 
Denote those summands by $U$ and the rest by $W$.
Then $W$ satisfies the hypothesis of Lemma \ref{lem:decomposition extension} and we have a decomposition of $W$ already.
In particular, we can write $W\cong W_1\oplus W_2$ where the summands of $W_1$ are nonzero at $s_{l-1}$ and the summands at $W_2$ are 0 at $s_{l-1}$.
Then by a further use of Lemma \ref{lem:W(b)=0 implies W is summand} we see $W_2$ is actually the summand of $V$ that we desired.
\end{proof}

\begin{notation}\label{note:V0infinity}
Let $A_\R$ and $V$ be as in Construction \ref{con:V0infinity}.
For some $l$, let $W_2$ be as in the end of the proof of Lemma \ref{lem:split subrepresentation}.
As seen in the proof, $W_2$ is a direct sum of interval indecomposables.
Let $V^l_{0,\infty}$ be the direct sum of those summands of $W_2$ who have support at $s_l$.
\end{notation}

\begin{remark}\label{rem:V0infinity dual} 
Construction \ref{con:V0infinity}, Proposition \ref{prop:V1n is split}, Lemma \ref{lem:decomposition extension}, Lemma \ref{lem:split subrepresentation}, and Notation \ref{note:V0infinity} can all be performed on $(-\infty,s_{l+1}]$ instead of $[s_{l-1},+\infty)$. 
These representations will be denoted $V^{i,n}_l$ and $V^{0,\infty}_l$.
\end{remark}

\begin{lemma}\label{lem:step one of straight}
 Let $V$ be a pwf representation with support on an open interval which has no sinks or sources. Let $c$ be any point in the open interval. Then $V=V_0\oplus V_1\oplus V_2$ where $V_0$ is a direct sum of finitely many interval representations having $c$ in its support and $V_1$, $V_2$ are representations having support strictly below $c$, strictly above $c$, respectively.
\end{lemma}
 
\begin{proof}
This follows from a combination of Lemmas \ref{lem:Y} and \ref{lem:lem A}.
\end{proof}
 
In order to prove Theorem \ref{thm:indecomposables} we prove the following lemma, which recovers the specific case of the real line in Crawley-Boevey's theorem in \cite{Crawley-Boevey2015}.

\begin{lemma}\label{lem:straight}
Let $V$ be as in Lemma \ref{lem:step one of straight}. Then $V$ is a direct sum of interval representations.
\end{lemma}
 
\begin{proof}
Let $(a,b)$ be an open subinterval of $\R$ containing the support of $V$ where $a$ and $b$ are allowed to be $-\infty$ and $+\infty$, respectively.
Let $C$ be the set of all points in $(a,b)$ of the form $c(k,n):=\tan(\tan^{-1}a+k(\tan^{-1}b-\tan^{-1}a)/2^n)$ for positive integers $k,n$ where $k$ is odd, which is a dense subset of $(a,b)$.

By induction on $n$ using Lemma \ref{lem:step one of straight}, we get a decomposition of $V$ in the form $V= V_\infty\oplus \bigoplus V_{k,n}$ where each $V_{k,n}$ is a direct sum of finitely many interval indecomposables having $c(k,n)$ in its support and no point in the form $c(j,m)$ in its support where $m<n$ and $V_\infty$ has no elements of the form $c(k,n)$ in its support.
In that case, $V_\infty$ is a direct sum of simple representations since for any $a\le c<d\le b$, there is a number of the form $c(k,n)$ in the open interval $(c,d)$.
So, the morphism $V_\infty(c,d):V_\infty(c)\to V_\infty(d)$ must be zero since it factors though $V_\infty(c(k,n))=0$.
\end{proof} 

\begin{theorem}\label{thm:indecomposables}\label{thm:GeneralizedBarCode}
Let $A_\R$ be a continuous quiver of type $A$ and $V$ be a representation in $\ReppwfAR$.
Then $V$ is the direct sum of interval indecomposables (Definition \ref{def: MI}).
\end{theorem}
\begin{proof}
\textbf{Outline}: We complete this proof in four parts.
In Part 1, we consider the indecomposable summands whose support is contained entirely between a sink and source.
In Part 2, we consider the indecomposable summands whose support contains at least one but only finitely many sinks and sources.
In Part 3, we consider the indecomposable summands whose support may contain infinitely many sinks and sources, but is bounded on exactly one side.
Finally, in Part 4, we concern ourselves with indecomposable summands whose support is $\R$.
Since the case where $A_\R$ has no sinks or sources in $\R$ has been covered by Lemma \ref{lem:straight}, we assume that $A_\R$ has at least one sink or source in $\R$.

\textbf{Part 1}:
Let $s_n$ and $s_{n+1}$ be an adjacent pair of sink, source or $\pm\infty$; however, only one may be $\pm \infty$ by assumption.
We use the notation $[s_n,s_{n+1}]$ even if one of the endpoints is actually $\pm \infty$.
By Lemma \ref{lem:lem A}, $V_{[s_n,s_{n+1}]}$ decomposes to $A_n\oplus B_n$ where the support of $A_n$ is contained in $(s_n,s_{n+1})$.
By Lemma \ref{lem:W(b)=0 implies W is summand}, $A_n$ is a direct summand of $V$.

Thus, for all $n$ where $s_n$ or $s_{n+1}$ is in $\R$, we have such an $A_n$.
So we have that $V\cong (\bigoplus A_n)\oplus U$.
By Lemma \ref{lem:straight} each $A_n$ decomposes into a direct sum of indecomposable representations.

\textbf{Part 2}: We now assume $V\cong U$ as in the end of Part 1.
If $A_\R$ has finitely many sinks and sources then, by the proof of Lemma \ref{lem:lem C}, $V$ is a finite direct sum of indecomposable representations.
So we shall now assume $A_\R$ has infinitely many sinks and sources.
Choose a sink or source $s_l$ in $\R$.
By Proposition \ref{prop:V1n is split} we know $V^l_{i,n}$ is a split subrepresntation with bounded support for all $1\leq i \leq n$.~
Thus, for all $l$ such that $s_l\in \R$ we obtain such direct summands, none of which are counted twice.
Thus we have $V\cong ( \bigoplus_{l,n} V^l_{1,n})\oplus U$.

\textbf{Part 3}: Now we assume $V\cong U$ as in the end of Part 2. 
Then for each $l\in\Z$ we apply Lemma \ref{lem:split subrepresentation} and obtain $V^l_{0,\infty}$ as in Notation \ref{note:V0infinity}.
By Remark \ref{rem:V0infinity dual} we also obtain $V_l^{0,\infty}$ for each $l$.
Each $V_{0,\infty}^l$ and $V_l^{0,\infty}$ decompose into interval indecomposables and so we have
$V\cong (\bigoplus (V_{0,\infty}^l\oplus V_l^{0,\infty}))\oplus U$.

\textbf{Part 4}: We assume $V\cong U$ as in the end of Part 3.
For any $s_l \in \R$, we know $V^l_{0,\infty}=0$ and $V_l^{0,\infty}=0$. 
Choose some sink or source $s_l$ in $\R$ and let $X=V_{[s_l,+\infty)}$ and $Y= V_{(-\infty,s_l]}$. 
We can then construct $X^l_{0,\infty}$ and $Y_l^{0,\infty}$.
Since $V^l_{0,\infty}=0$ and $V_l^{0,\infty}=0$, we see $\dim X^l_{0,\infty}(s_l) = \dim Y_l^{0,\infty}(s_l)$.
In particular, they are both finite.

Furthermore, $V(x,y)$ is an isomorphism for all $y\preceq x$ in $\R$.
Choose a decomposition of $V_{[s_{l-1},s_{l+1}]}$ and use the technique in Lemma \ref{lem:decomposition extension} to extend this decomposition to all of $X$ and all of $Y$.
But together this yields a decomposition of $V$.~

This will give us a bijection $V\cong \bigoplus_{\dim V(s_l)} M_{(-\infty,+\infty)}$.
Thus, $V$ is a direct sum of indecomposable representations.

\textbf{Conclusion}: In Parts 1--3 we decomposed $V$ into $Z\oplus U$ and in Parts 2--4 we decomposed the previous Part's $U$.
In Parts 1--3 we showed that the $Z$ summand was a direct sum of indecomposables and in Part 4 we showed the final $U$ is a direct sum of indecomposables.
Therefore, given any pointwise finite-dimensional representation $V$ of $A_\R$, it is the direct sum of indecomposable representations.
If $V$ itself is indecomposable it appears as one of described indecomposable summands, depending on its support.
\end{proof}

\begin{remark}\label{rem:indecomposableprojectives}
The theorem above, with the aid of Theorem \ref{thm:characterization of one sided projectives}, completely classifies indecomposable projective objects in $\ReppwfAR$ and $\RepbAR$.
They come in three forms, up to isomorphism.
\begin{enumerate}
\item $P_a$ as in Definition \ref{def:projective generated}:
\begin{align*}
P_a(x) &= \left\{\begin{array}{ll} k & x\preceq a \\ 0 & \text{otherwise} \end{array}\right. & 
P_a(x,y) &= \left\{\begin{array}{ll} 1_k & y\preceq x \preceq a \\ 0 & \text{otherwise}. \end{array}\right.
\end{align*} 
\item $P_{a)}$ given by
\begin{align*} P_{a)} &= \left\{ \begin{array}{ll}k & x\preceq a, x<a \\ 0 & \text{otherwise} \end{array}\right. & 
P_{a)}(x,y) &= \left\{\begin{array}{ll} 1_k & y\preceq x \preceq a, y\leq x < a \\ 0 & \text{otherwise}. \end{array}\right. \end{align*}
\item $P_{(a}$ given by
\begin{align*} P_{(a} &= \left\{ \begin{array}{ll}k & x\preceq a, a<x \\ 0 & \text{otherwise} \end{array}\right. & 
P_{(a}(x,y) &= \left\{\begin{array}{ll} 1_k & y\preceq x \preceq a, a<x\leq y \\ 0 & \text{otherwise}. \end{array}\right.\end{align*}
\end{enumerate} 
Note that unless $a$ is a source at least one of (2) or (3) will define the 0 representation.
If $a$ is a sink then both (2) and (3) will be the 0 representation.

 Additionally, it is worth noting that if $V$ is a subrepresentation of any sum of projectve indecomposables then $V$ is also projective.
This follows from Theorem \ref{thm:characterization of one sided projectives} (1).
Therefore, $\ReppwfAR$ is hereditary.~
\end{remark}

\begin{example}
Let the set of sinks and sources $S=\{0,1\}$, where $s_0=0$ is a sink and $s_1=1$ is a source.
We provide a complete list of indecomposable projectives in $\ReppwfAR$ with this orientation. 
The values $a,b,c\in\R$ below are such that $a<0<b<1<c$. {\begin{center} \begin{tabular}{ccccccccccc} \{ $\{0\}$, & $(-\infty,0]$, & $(a,0]$, & $[a,0]$,  & $[0,b)$, & $[0,b]$, & $[0,1)$, & $[0,+\infty)$, &  $(1,+\infty)$, & $(c,+\infty)$, &  $[c,+\infty)$  \} \\
$P_0$ & $P_{-\infty}$ & $P_{(a}$ & $P_a$ & $P_{b)}$ & $P_b$ & $P_{1)}$ & $P_1$ & $P_{(1}$ & $P_{(c}$ & $P_c$ \end{tabular} \end{center}} 
\end{example}

\begin{remark}\label{rem:indecomposableinjectives}
We also have the indecomposable injective objects in $\ReppwfAR$.
\begin{enumerate}
\item $I_a$ given by:
\begin{align*}
I_a(x) &= \left\{\begin{array}{ll} k & a\preceq x \\ 0 & \text{otherwise} \end{array}\right. & 
I_a(x,y) &= \left\{\begin{array}{ll} 1_k & a\preceq y\preceq x \\ 0 & \text{otherwise} \end{array}\right.
\end{align*} 
\item $I_{a)}$ given by
\begin{align*} I_{a)} &= \left\{ \begin{array}{ll}k & a\preceq x, x<a \\ 0 & \text{otherwise} \end{array}\right. & 
I_{a)}(x,y) &= \left\{\begin{array}{ll} 1_k & a \preceq y\preceq x, x\leq y< a \\ 0 & \text{otherwise} \end{array}\right. \end{align*}
\item $I_{(a}$ given by
\begin{align*} I_{(a} &= \left\{ \begin{array}{ll}k & a\preceq x, a<x \\ 0 & \text{otherwise} \end{array}\right. & 
I_{(a}(x,y) &= \left\{\begin{array}{ll} 1_k & a\preceq y\preceq x, a<y\leq x \\ 0 & \text{otherwise} \end{array}\right.\end{align*}
\end{enumerate} 
\end{remark}

\subsection{More on $P_{(a}$, $P_{a)}$, and the Pointwise Finite Requirement}\label{sec:pwf requirement}

As mentioned in Section \ref{sec:projectives}, the indecomposable projectives $P_{(a}$ and $P_{a)}$, whichever are nonzero, are not projective in $\RepAR$. They are only projective in the smaller subcategory $\ReppwfAR$.
We will prove this using a specific representation, denoted $\mathfrak P$, that exists only in $\RepAR$.
I.e., it is not pointwise finite-dimensional.
We will use that same representation to show why Theorem \ref{thm:indecomposables} can fail without the pointwise finite assumption.

\begin{construction}\label{con:bad representation}
We will denote the problematic representation by $\mathfrak P$.
First, let $a\in \R$ such that $a$ is not a sink.
Let $p\in \R$ such that $p\preceq a$ and $p\neq a$.
By symmetry, suppose $p<a$. 
Let $\{x_i\}_{i=0}^\infty$ be a strictly increasing sequence converging to $a$ such that $x_0> p$.
Let $M = \bigoplus_{\{x_i\}} M_{[p,x_i]}$. 
Then the support of $M$ is $[p,a)$. 
Let $\pi : M(p) \to k$ be a surjection given by sending each $1$ in $M_{[p,x_i]}(p)=k$ to $1\in k$.

Let $\mathfrak P$ be given by
\begin{align*}
\mathfrak P(x) &= \left\{ \begin{array}{ll} k & x=p \\ M(x) & x\neq p\end{array} \right. &
\mathfrak P(x,y) &= \left\{ \begin{array}{ll} 1_k & x=y=p \\ \pi\circ M(x,y) &x\neq y=p \\ M(x,y) & \text{otherwise}. \end{array}\right.
\end{align*} 
We see that $\mathfrak P$ also has support $[p,a)$. 
\hfill $\diamond$
\end{construction}

\begin{proposition}\label{prop:no bad maps}
Let $A_\R$ be a continuous quiver of type $A$.
Let $p,a\in \R$ such that $a$ is not a sink, $p\preceq a$, and $p< a$. 
Then there is no nontrivial morphism $P_{a)}\to \mathfrak P$, where $\mathfrak P$ is from Construction \ref{con:bad representation}.
\end{proposition}
\begin{proof}
Choose $x_m$ in the sequence from Construction \ref{con:bad representation}. 
Let $f(x_m):P_{a)}\to \mathfrak P(x)$ be a linear map. 
Since $P_{a)}=k$, $f(x_m)$ is determined by $f(x_m)(1)$.
Since $\mathfrak P(x) = M(x)$ for $x\neq p$, we see
\begin{displaymath} f(1) = (\underbrace{0,\ldots ,0}_{m-1}, r_m, r_{m+1},\ldots, r_n,0,0,\ldots)\end{displaymath} 
Then for any linear map $f(x_{n+1}):P_{a)}(x_{n+1})\to \mathfrak P(x_{n+1})$ we know that 
\begin{displaymath}
f(x_m)\circ P_{a)}(x_{n+1},x_m) \neq \mathfrak P(x_{n+1},x_m) \circ f(x_{n+1})
\end{displaymath} 
Therefore, there is no morphism of representations $P_{a)}\to \mathfrak P$.
\end{proof}

\begin{proposition}\label{prop:not projective in big category}
Let $A_\R$ be a continuous quiver of type $A$ and $a\in \R$ such that $a$ is not a sink.
Then each nonzero $P_{(a}$ and $P_{a)}$ is not projective in $\RepAR$.
\end{proposition}
\begin{proof}
Let $p\in\R$ such that $p< a$ and $p\preceq a$.
The other case, where $p>a$ and $p\preceq a$, is similar. 
Then, there is a nontrivial morphism of indecomposable representations $f:P_{a)}\to M_{[p,a)}$. 
Let $\mathfrak P$ be as in Construction \ref{con:bad representation}. 
For each $x\in [p,a)$, let $f(x): \mathfrak P(x) \to M_{[p,a)}(x)$ be given by $\mathfrak P(x,p)$. 
Since $\mathfrak P(p) = M_{[p,a)}(x)$ for all $x\in[p,a)$, this is a well-defined morphism of representations. 
In particular, it is an epimorphism. 

So now we have an epimosphism $P_{a)}\twoheadrightarrow M_{[p,a)}$ and an epimorphism $\mathfrak P\twoheadrightarrow M_{[p,a)}$. 
However, there is no nontrivial morphism $P_{a)}\to \mathfrak P$ in $\RepAR$, by Proposition \ref{prop:no bad maps}. 
Therefore, $P_{a)}$ is not projective.
\end{proof}

\begin{proposition}\label{prop:bad representation is bad}
Let $A_\R$ be a continuous quiver of type $A$ and $\mathfrak P$ as in Construction \ref{con:bad representation}.
Then $\mathfrak P$ is not the direct sum of pointwise one-dimensional indecomposables.
\end{proposition}
\begin{proof} 
We saw in the proof of Proposition \ref{prop:not projective in big category} that there is an epimorphism $\mathfrak P\twoheadrightarrow M_{[p,a)}$. 
However, just as in the proof of Proposition \ref{prop:no bad maps} there are no nontrivial morphisms $M_{[p,a)}\to \mathfrak P$. 
Thus, $M_{[p,a)}$ is not a direct summand of $\mathfrak P$. 
But if $\mathfrak P$ had a direct sum decomposition, one of the components must have support $[p,a)$. 
But that would mean the indecomposable is $M_{[p,a)}$. 
Therefore, $\mathfrak P$ does not decompose into a direct sum of one-dimensional indecomposables.
\end{proof}
 
\subsection{Relation to Decomposition Theorems in Persistent Homology}\label{sec:relation to BC-B}

Theorem \ref{thm:indecomposables} is, in some sense, a combination of the Crawley-Boevey's BarCode theorem from \cite{Crawley-Boevey2015} and Botnan's decomposition theorem in \cite{Botnan}.
Part of our argument actually follows the latter paper.
The BarCode theorem handles representations on the continuum but only a straight orientation.
By contrast, Botnan's decomposition handles the infinite zigzag orientation but only in the discrete setting.

One might think to use Botnan's paper explicitly with Crawley-Boevey's result.
However, this cannot be done directly.
In order to make use of the combination of theorems, several technical lemmas would still be required.
In particular, one would have to argue which pwf representations can be ``lifted'' to a discrete quiver and then prove that the decomposition can be ``pushed back down.''
While intuitive, the technical details in such an argument (see \cite{HansonRock} for a similar argument) are still involved.
We avoided such a proof in order to provide a self-contained foundation of continuous type $A$ quivers as well as an algorithmic proof of Theorem \ref{thm:GeneralizedBarCode}.

While Theorem \ref{thm:indecomposables} recovers a result by Botnan and Crawley-Boevey in \cite{BotnanCrawley-Boevey}, the method of proof is different.
One might consider the proof presented in Section \ref{sec:thetheorem} as a ``direct'' proof while the proof in \cite{BotnanCrawley-Boevey} uses representations of products of posets.

Both of \cite{Crawley-Boevey2015, Botnan} worked with pointwise finite-dimensional representations and each displayed a non-example for a representation that is not pointwise finite-dimensional.
Theorem \ref{thm:indecomposables} adheres to exactly the same restrictions and a relevant non-example appears in Section \ref{sec:pwf requirement} as Construction \ref{con:bad representation} and Proposition \ref{prop:bad representation is bad}.

\section{Finitely Generated Representations: $\repAR$}\label{sec:little rep}

In this section we will prove results about the category of finitely generated representations, denoted $\repAR$.
Many of the properties one could reasonably expect to hold in a continuous version of $\text{rep}_k(A_n)$ do, in fact, hold for $\repAR$.
The properties that change due to the nature of the continuum are Auslander--Reiten sequences and descending chains of subrepresentations.
We provide an incomplete list of the properties that hold or do not hold in the form of a theorem and dedicate the rest of this section to proving each of the items in the theorem.

\begin{theorem}\label{thm:little rep}
Let $A_\R$ be a continuous quiver of type $A$ and denote by $\repAR$ the category of finitely generated representations (Definition \ref{def:finitelygeneratedreps}).
Then the following hold.
\begin{enumerate}
\item For indecomposable representations $M_I$ and $M_J$ in $\ReppwfAR$, $\RepbAR$, or $\repAR$, we have $\Hom(M_I,M_J)\cong k$ or $\Hom(M_I,M_J)=0$ (Proposition \ref{prop:homiskor0}).
\item Every morphism $f:V\to W$ in $\ReppwfAR$, $\RepbAR$, or $\repAR$ has a kernel, a cokernel, and coinciding image and coimage in that category. (Lemma \ref{lem:abelian})
\item The category $\repAR$ Krull-Schmidt, \emph{but not} artinian (Lemma \ref{lem:krull-schmidt}, Proposition \ref{prop:not artinian}).
\item The global dimension of $\repAR$ is 1 (Proposition \ref{prop:projectiveresolution}).
\item The $\Ext$ space of two indecomposables $M_I$ and $M_J$ in $\ReppwfAR$, $\RepbAR$, or $\repAR$ is either isomorphic to $k$ or is 0 (Proposition \ref{prop:smallext}).
\item While some Auslander--Reiten sequences exist (Proposition \ref{prop:ARexistence}), some indecomposables have \emph{neither} a left \emph{nor} a right Auslander--Reiten sequence (Proposition \ref{prop:no AR sequences}).
\end{enumerate}
\end{theorem}

\subsection{Requisites and Definition}

In this subsection we define the category of finitely generated representations of a continuous type $A$ and prove Theorem \ref{thm:little rep} (1) -- (3).

\begin{notation}
We may use $|$ instead of $($, $)$, $[$, or $]$ to write an interval.
When this happens, we mean that the endpoint may or may not be included; either we are making no assumptions about endpoints or it is clear what choice is possible from context.
I.e., for all $a,b\in\R$, $|a,b|$ can be one of four possibilities.
However, when we write our intervals, we allow $a=-\infty$ and $b=+\infty$ so long as we obtain a subset of $\R$.
So, the notation $|a,b|$ will never mean $[-\infty,b|$, $|a,+\infty]$, or $[-\infty,+\infty]$.
\end{notation}

\begin{proposition}\label{prop:homiskor0}
Let $V$ and $W$ be indecomposable representations in $\ReppwfAR$.
Then either $\Hom(V,W)\cong k$ or $\Hom(V,W)=0$.
\end{proposition}
\begin{proof}
Suppose $\Hom(V,W)\neq 0$ and choose a nontrivial $f:V\to W$.
Then there is $x\in\R$ such that $f(x):V(x)\to W(x)$ is not 0.
Since $V(x)\cong k\cong W(x)$ we see $f(x)$ is an isomorphism.
For all $y\preceq x$, $W(x,y)\circ f(x)=f(y)\circ V(x,y)$.
If $V(y)\neq 0$ and $W(y)\neq 0$ then $f(y)= W(x,y)\circ f(x)\circ V(x,y)^{-1}$.
For all $z$ such that $x\preceq z$, $W(z,x)\circ f(z)=f(x)\circ V(z,x)$.
Then again if the vector spaces are nontrivial we have $f(z)=W(z,x)^{-1}\circ f(x)\circ V(z,x)$.

So for the sink and source $s\preceq x\preceq s'$ we see each of $f(s)$ and $f(s')$ are either 0 or determined by $x$.
Since the set of sinks and sources is discrete with no accumulation points we can use our arguments in the previous paragraph repeatedly and see that each nontrivial $f(y)$ is determined by $f(x)$.
Since $\Hom(V(x),W(x))\cong k$ and every nontrivial $f(y)$ is determined by $f(x)$, we see $\Hom (V,W)\cong k$.
\end{proof}

\begin{deff}\label{def:finitelygeneratedreps}
We define $\repAR$ as the full subcategory of $\ReppwfAR$ whose objects are representations $V$ that are finitely generated by indecomposable projectives (listed in Remark \ref{rem:indecomposableprojectives}).
\end{deff}

\begin{lemma}\label{lem:abelian} Let $f:V\to W$ be a morphism in $\mathcal C$ where $\mathcal C=\ReppwfAR$, $\RepbAR$, or $\repAR$.
\begin{itemize}
\item $f$ has a kernel in $\mathcal C$,
\item $f$ has a cokernel in $\mathcal C$, and
\item the image and coimage of $f$ coincide and lie in $\mathcal C$.
\end{itemize}
\end{lemma}

\begin{proof}

First note that $f$ is a morphism in $\RepAR$.
By a dimension argument for $V(x)$, $W(x)$, $\ker f(x)$, and $\coker f(x)$ at each $x\in \R$ the statement must be true for $\mathcal C=\ReppwfAR$ and $\mathcal C=\RepbAR$.

Now suppose $\mathcal C=\repAR$.
Since $\ReppwfAR$ is abelian the image and coimage of $f$ coincide.
Since $V\twoheadrightarrow \im f$ and $V$ is finitely generated, so is $\im f$.
Similarly, since $W$ is finitely generated by some $\bigoplus_{i=1}^n P_i$ there is a surjection $\bigoplus_{i=1}^n P_i\twoheadrightarrow \coker f$.

Suppose $g:\bigoplus Q_i\twoheadrightarrow V$ generates $V$.
Then $\ker(f\circ g)$ is a subrepresentation of a projective; since $\ReppwfAR$ is hereditary this means $\ker(f\circ g)$ is projective.
Also $\ker(f\circ g)$ maps to $\ker f$.
For any $0\neq \hat{v}\in \ker f(x)$ there is $v\in V(x)$ from the inclusion.
Then there is $\tilde{v}\in \bigoplus Q_i(x)$ that maps to $v$.

Let $\bigoplus Q_i'=\ker (f\circ g)$.
Any projective subrepresentation of a finitely generated projective is finitely generated, so $\bigoplus Q_i'$ is finitely generated.
We also know that since $\tilde{v}\mapsto v\mapsto 0$, there exists $\bar{v}\in \bigoplus Q_i'(x)$ that maps to $v$ and so maps to $\hat{v}$.
Thus, $\bigoplus Q_i'\twoheadrightarrow \ker f$ so $\ker f$ is also finitely generated.
Therefore, $\ker f$, $\im f$, and $\coker f$ are all generated by finitely generated and so in $\repAR$.
\end{proof}

\begin{lemma}\label{lem:krull-schmidt}
Let $V$ be a representation in $\repAR$.
Then $V$ is isomorphic to a finite direct sum of interval indecomposables.
Furthermore, $\repAR$ is Krull-Schmidt.
\end{lemma}
\begin{proof}
Suppose $V$ is in $\repAR$ and $\bigoplus_{i=1}^n Q_i\to V$ be a surjective morphism required by Definition \ref{def:finitelygeneratedreps}.

Since $\dim Q_i(x) \leq 1$ for all $x\in \R$, $\dim V(x) \leq n$ for all $x\in \R$.
That is, both $Q$ and $V$ are in $\RepbAR$.
By Theorem \ref{thm:indecomposables}, $V$ is a direct sum of (\emph{a priori} possibly infinitely many) interval indecomposables.

Since each $Q_i$ is projective, the support of each $Q_i$ contains at most 3 sinks and sources (1 source and 2 sinks).
Then, since $Q$ is a finite direct sum, the support of $Q$ itself contains finitely many sinks and sources.
Since $Q$ surjects onto $V$, the support of $V$ must also contain only finitely many sinks and sources.

For contradiction, suppose $V$ is an infinite direct sum of indecomposables.
Since $V$ is pointwise finite-dimensional and its support contains finitely many sinks and sources, infinitely many summands must have support that does not contain a sink or a source; i.e.~each of these indecomposable's support is bounded by an adjacent sink and source.
Since there are only finitely many sinks and sources in the support of $V$, infinitely many must have support between the same adjacent sink and source.

For each $Q_i = P_a$ for some $a$ (classification in Remark \ref{rem:indecomposableprojectives}), any indecomposable hit by $Q_i$ must contain $a$ in its support.
Since $V$ is pointwise finite dimensional there can only be finitely many such indecomposables.
Thus there must be some $Q_i = P_{(a}$ or $P_{a)}$.

If $Q_i=P_{(a}$ then any indecomposable $V_\alpha$ hit by $Q_i$ has the property that $glb\supp V_\alpha \leq a$.
If $Q_i$ hit infinitely many indecomposables there must be infinitely many with support of the form $(a,b_\alpha)$ and the $b_\alpha$ must converge on $a$.
However, $V$ is also in $\RepbAR$ and so this is a contradiction as $\lim \dim V(x)$ as $x\to a$ from above would $\infty$.
The same argument holds if $Q_i=P_{a)}$.~
Therefore, $V$ is the direct sum of finitely many indecomposables.
Combined with Theorem \ref{thm:iso-indecomps} and Lemma \ref{lem:V is schurian} this shows $\repAR$ is Krull-Schmidt.
\end{proof}

\begin{remark}
In \cite{SalaSchiffmann2019}, Sala and Schiffmann prove their category of coherent representations (which they call coherent persistence modules) has similar properties to Theorem \ref{thm:little rep}.
In their paper, \emph{tame} representations have finitely-many places where non-isomorphisms occur in the representation. \emph{Coherent} representations are tame with bounded support and with right continuous dimension functions. Our category $\repAR$ is the category of tame representations of $A_\R$.
\end{remark}

\begin{proposition}\label{prop:not artinian}
The category $\text{rep}_k(A_{\R})$ is not Artinian.
\end{proposition}
\begin{proof}
Let $P_a$ be a projective indecomposable (Remark \ref{rem:indecomposableprojectives}) such that $a$ is not in $S$.
Let $b\in\bar{S}$ such that $b\preceq a$; note $b\neq a$.
Then, for every $b \preceq z \preceq a$ such that $b\neq z \neq a$, $P_z\subsetneq P_a$.
Furthermore, for any two such $z,z'$ such that $z\preceq z'$, we have  $P_z\subsetneq P_{z'}\subsetneq P_a$.
Thus, we have an infinite (uncountable!) descending chain and so $\repAR$ is not Artinian.
\end{proof}

\begin{example}\label{xmp:notfinitelygeneratedrep}
Let us return to the representation $M$ in Example \ref{xmp:boundedreps}.
It is an uncountable sum and so not in the category $\repAR$.
In particular, any surjection onto $M$ by a sum of interval indecomposables would require the source representation to be an uncountable sum as well.
\end{example}

\subsection{Properites of $\repAR$}\label{sec:properties of rep}
We now prove Theorem \ref{thm:little rep} (4) and (5).

\begin{proposition}\label{prop:doubleinfiniteisthesame}
Let $A_\R$ and $A'_\R$ be different orientations such that the sinks and sources are unbounded above and below in both $A_\R$ and $A'_\R$.
Then $\repAR\cong \text{rep}_k(A'_\R)$.
\end{proposition}
\begin{proof}
We'll define a bijection $F:\R\to\R$ that induces a bijection on (isomorphism classes of) indecomposables and thus an equivalence of categories.
Recall $S$ is the set of sinks and sources of $A_\R$ and $S'$ is the set of sinks and sources of $A'_\R$.
First define the bijection on $S\to S'$ to be $s_n\mapsto s'_n$.
Let $x\in\R$ and $n\in\Z$ such that $s_n<x<s_{n+1}$.
Then $x=t\cdot s_n + (1-t) s_{n+1}$ for some $t\in(0,1)$.
Let $F(x)= t\cdot F(s_n)+(1-t) F(s_{n+1})$.

This induces a bijection on indecomposables as it is a bijection on $\R$.
In particular, if $x\preceq y$ then $F(x)\preceq F(y)$.
If $\Hom(M_{|a,b|},M_{|c,d|})\cong k$ in $\repAR$ then $a\preceq c$ and $b\preceq d$.
Since $F(a)\preceq F(c)$ and $F(b)\preceq F(d)$, the $\Hom$-set from $M_{|F(a), F(b)|}$ to  $M_{|F(c), F(d)|}$ is also isomorphic to $k$.
Thus we have an equivalence on the indecomposables.
Since both categories are Krull-Schmidt we have an equivalence of categories.
\end{proof}

\begin{proposition}\label{prop:projectivesmaps}\label{prop:injectivesmaps}
Let $P$ and $Q$ be projective indecomposables in $\repAR$ and $I$ and $J$ be injective indecomposables in $\repAR$.
\begin{itemize}
\item Any morphism $f:P\to Q$ is either 0 or mono.
\item Any morphism $g:I\to J$ is either 0 or epi.
\end{itemize}
\end{proposition}
\begin{proof}
We will prove the first statement; the second is dual.
Let $f:P\to Q$ be a map of indecomposable projectives.
By Theorem \ref{thm:characterization of one sided projectives} and Remark \ref{rem:indecomposableprojectives} the image $\im f$ in $Q$ is a subrepresentation and so projective.
Since $P$ surjects on to $\im f$ it is a split subrpresentation of $P$.
However, $P$ is indecomposable so $\im f = 0$ or $\im f \cong P$.
\end{proof}

Below, for each indecomposable representation $V$ in $\repAR$ we create two projective representations $P_0(V)$ and $P_1(V)$.
In Proposition \ref{prop:projectiveresolution} we prove that $P_1(V)\to P_0(V)\to V$ is the minimal projective presentation of $V$.
\begin{construction}\label{con:soVsiV}
Let $V$ be an indecomposable in $\repAR$ with support $|a,b|$.
If $V$ is projective let $P_0(V)=V$ and $P_1(V)=0$.

Now suppose $V$ is not projective.
Recall $S$ is the set of sinks and sources of $A_\R$ in $\R$.
Since $V$ is finitely generated $|a,b|\cap S$ is finite. 
We let $P_0(V)$ be the direct sum of the following indecomposable projectives.
\begin{itemize}
\item $P_s$ for all sources $s$ in $(a,b)$.
\item $P_{(a}$ if $a\notin |a,b|$ and there exists $x\preceq a$ in $|a,b|$.
\item $P_a$ if $a\in |a,b|$ and there exists $x\preceq a$, $x\neq a$ in $|a,b|$.
\item $P_{b)}$ if $b\notin |a,b|$ and there exists $x\preceq b$ in $|a,b|$.
\item $P_b$ if $b\in |a,b|$ and there exists $x\preceq b$, $x\neq b$ in $|a,b|$.
\end{itemize}
We let $P_1(V)$ be the direct sum of the following indecomposable projectives.
\begin{itemize}
\item $P_s$ for all sources $s$ in $(a,b)$.
\item $P_a$ if $a\notin |a,b|$ and there exists $a\preceq x$ in $|a,b|$. 
\item $P_{a)}$ if $a\in |a,b|$.
\item $P_b$ if $b\notin |a,b|$ and there exists $b\preceq x$ in $|a,b|$.
\item $P_{(b}$ if $b\in |a,b|$. 
\end{itemize}
If $a$ or $b$ is a sink and in $|a,b|$ then the summand $P_{a)}$ or $P_{(b}$ is 0, respectively. 
We see that both $P_0(V)$ and $P_1(V)$ are nontrivial and finitely generated, so in $\repAR$.
\hfill $\diamond$
\end{construction}

\begin{proposition}\label{prop:injectivemorphismonprojectives}
Let $V$, $P_1(V)$, and $P_0(V)$ be as in Construction \ref{con:soVsiV}.
Then there is an injective morphism $P_1(V)\hookrightarrow P_0(V)$ whose cokernel is $V$.
\end{proposition}
\begin{proof}
If $V$ is projective the statement is trivially true.
Now suppose $V$ is not projective.
There are finitely many sinks and sources, totally ordered.
So on those summands we let the maps be defined in the following way where $\pm$ means scalar multiplication by $\pm 1$:\begin{displaymath}\xymatrix@C=1ex{
& \cdots \ar[dr]^-{-} & & P_{s_{2n}} \ar[dl]_-{+} \ar[dr]^-{-} & & P_{s_{2n+2}} \ar[dl]_-{+} \ar[dr]^-{-}& & {\cdots} \ar[dl]_-{+} \ar[dr]^-{-}& & P_{s_{2n+2m}} \ar[dl]_-{+} \ar[dr]^-{-}& & \cdots  \ar[dl]_-{+} \\
\cdots & & P_{a*}\text{ or }P_{s_{2n-1}} & & P_{s_{2n+1}} & & P_{s_{2n+3}} & & \cdots & & P_{s_{2n+2m+1}} & & \cdots 
}\end{displaymath}

Since there is no accumulation of elements of $S$ in $\R$, a projective indecomposable at $a$ can only appear as a summand of $P_0(V)$ or $P_1(V)$, but not both.
The similar statement is true for $b$.
Thus, only one type of projective summand of each $a$ or $b$ may appear in $P_0(V)$ and $P_1(V)$.
Denote whichever summands appear, if any, by $P_{a*}$ and $P_{b*}$.

If $P_{a*}$ appears in $P_1(V)$ then there is a nontrivial map from $P_{a*}$ to $P_{s_{2n+1}}$ or $P_{b*}$, depending on whether or not $(a,b)$ contains any sources.
If this is the case, use scalar multiplication by $-1$.
In the similar case for $b$, use scalar multiplication by $+1$.

If $P_{a*}$ appears in $P_0(V)$ then there is a nontrivial map from $P_{s_{2n}}$ or $P_{b*}$ to $P_{a*}$, depending on whether or not $(a,b)$ contains any sinks.
If this is the case, use scalar multiplication by $+1$.
In the similar case for $b$, use scalar multiplication by $-1$.

Instead of proving that this map is injective with cokernel $V$, we instead note that the kernel of the surjection $P_0(V)\twoheadrightarrow V$ is $P_1(V)$.
This is equivalent.
\end{proof}

\begin{proposition}\label{prop:projectiveresolution}
The following hold:
\begin{itemize}
\item For any indecomposable $V$ in $\repAR$, $P_1(V)\hookrightarrow P_0(V)\twoheadrightarrow V$ is the minimal projective resolution and presentation of $V$.
\item All representations in $\repAR$ are finitely presented.
\item The global dimension of $\repAR$ is 1.
\end{itemize}
\end{proposition}
\begin{proof}
We see $P_1(V)$ is superflous in $P_0(V)$ and $P_1(V)\hookrightarrow P_0(V)\twoheadrightarrow V$ is exact by Proposition \ref{prop:injectivemorphismonprojectives}.
Thus the sequence is the minimal projective resolution and presentation of $V$.
Furthermore, noting that the reversal of orientation $\preceq$ on $\R$ gives the opposite category, we see the global dimension of $\repAR$ is 1.
\end{proof}

\begin{proposition}\label{prop:smallext}
Let $V$ and $W$ be indecomposables in $\repAR$.
If $\Ext^1(W,V)\neq 0$ then $\Ext^1(W,V)\cong k$.
\end{proposition}
\begin{proof}
Let $V$ and $W$ be indecomposables in $\repAR$.
By Proposition \ref{prop:projectiveresolution} the projective resolution of $V$ is $P_1(V)\hookrightarrow P_0(V)\twoheadrightarrow V$.
By definition $\Ext^i(V,W)$ is the $i$th homology group in the chain
\begin{displaymath}\xymatrix{
0 \ar[r] & \Hom(P_0(V),W) \ar[r] & \Hom(P_1(V),W) \ar[r] & 0.
} \end{displaymath}
Suppose $\Ext^1(W,V)\neq 0$.

Index the projectives in $P_0(V)$ that nontrivially map to $W$ from 1 to $m$, denoted $P_1,\ldots,P_m$, such that if $P_a=P_i$ and $P_b=P_{i+1}$ for $a,b\in\R$ then $a<b$.
Then $\Hom(P_0(V),W)\cong k^m$.
Let $f:(x_1,\ldots,x_m)$ be a nontrivial map $P_0(V)\to W$ and $\iota:P_1(V)\to P_0(V)$ the inclusion.
Index the projectives in $P_1(V)$  that nontrivially map to $W$ from 1 to $n$, similarly to the projectives in $P_0(V)$, denoted $Q_1,\ldots, Q_n$.

Then $Q_1$ maps to $P_1$ and $P_2$ or just $P_1$. 
If $Q_1$ only maps to $P_1$ then the projective $Q_2$ maps to both $P_1$ and $P_2$.
If $Q_1$ maps to both $P_1$ and $P_2$ then $Q_2$ maps to  $P_2$ and $P_3$.
Thus, the composition $f\circ\iota$ will be one of four forms:
\begin{itemize} \item $(x_1,x_1\oplus x_2,\ldots,x_{i-1}\oplus x_i)$, \item $(x_1\oplus x_2,\ldots,x_{i-1}\oplus x_i,x_i)$,
\item $(x_1,x_1\oplus x_2,\ldots,x_{i-1}\oplus x_i,x_i)$, or \item $(x_1\oplus x_2,\ldots,x_{i-1}\oplus x_i)$. \end{itemize}
In any case, basic linear algebra shows us that $\Hom(P_0(V),W)\to \Hom(P_1(V),W)$ is surjective or injective and the difference in dimensions is either 0 or 1.
Therefore $\dim \Ext^1(W,V)$ is 0 or 1.
\end{proof}

\subsection{Existence of Some Auslander--Reiten Sequences}\label{sec:ARsequences}
In this subsection we will show that for any orientation of a continuous type $A$ quiver, the category $\repAR$ contains some Auslander--Reiten sequences but not all Auslander--Reiten sequences (Theorem \ref{thm:little rep} (6)).
However, we will not provide a complete classification of Auslander--Reiten sequences in this paper.
Such a classification will be provided in the sequel to this paper.

In \cite{GabrielRoiter}, Gabriel and Ro\u{\i}ter provide a general description of Auslander--Reiten sequences of representations of linear posets. However, a specific description to this context in the contemporary language and notation of representation theory is new.

We recall the definition of an almost-split sequence, commonly called an Auslander--Reiten sequence.
Such short exact sequences were originally defined by Auslander and Reiten in \cite{ARSequences}.
\begin{definition}\label{def:ARsequence} 
Let $\mathcal A$ be an abelian category and $0\to U \stackrel{f}{\to} V \stackrel{g}{\to} W\to 0$ a short exact sequence in $\mathcal A$.
The short exact sequence is an \underline{almost split sequence}, or \underline{Auslander--Reiten sequence} if the following conditions hold:
\begin{itemize}
\item $f$ is not a section and $g$ is not a retraction.
\item $U$ and $W$ are indecomposable.
\item If $h:U\to X$ is a nontrivial morphism of indecomposables and $U\not\cong X$ then $h$ factors through $f$.
\item If $h:X\to W$ is a nontrivial morphism of indecomposables and $X\not\cong W$ then $h$ factors through $g$.
\end{itemize}
\end{definition}

In the following proposition, recall that $S$ is the set of sinks and sources in a continuous quiver of type $A$ and that $\bar{S}$ includes $\pm \infty$.
\begin{proposition}\label{prop:ARexistence}
Let $s_n, s_{n+1}\in \bar{S}$ and $a,b\in\R$ such that $s_n<a<b<s_{n+1}$.
One of the following is a short exact sequence and in particular an Auslander--Reiten sequence.
\begin{itemize}
\item If $s_n$ is a sink then the Auslander--Reiten sequence is 
\begin{displaymath} \xymatrix{0\ar[r] &  M_{[a,b)} \ar[rr]^-{\left[\begin{array}{c}1\\1\end{array}\right]} && M_{[a,b]}\oplus M_{(a,b)} \ar[rr]^-{\left[\begin{array}{cc} 1 & -1 \end{array}\right]} && M_{(a,b]} \ar[r] & 0 } \end{displaymath}
\item If $s_n$ is a source then the Auslander--Reiten sequence is 
\begin{displaymath} \xymatrix{0\ar[r] &  M_{(a,b]}  \ar[rr]^-{\left[\begin{array}{c}1\\1\end{array}\right]}  && M_{(a,b)}\oplus M_{[a,b]} \ar[rr]^-{\left[\begin{array}{cc} 1 & -1 \end{array}\right]} && M_{[a,b)} \ar[r] & 0 }\end{displaymath} 
\end{itemize}
\end{proposition}
\begin{proof}
We note the two cases are symmetric and prove the first. 
We see the first map is injective, the second is surjective, and that the sequence is exact at $[a,b]\oplus(a,b)$. 
Thus, the sequence is a short exact sequence. 

Denote the map $M_{[a,b)}\to M_{[a,b]}\oplus M_{(a,b)}$ in the sequence by $h_1\oplus h_2$. 
By Proposition \ref{prop:sufficientProp} we know both $M_{[a,b)}$ and $M_{(a,b]}$ are indecomposable. 
Let $V$ be another indecomposable representation in $\repAR$.
By definition the support of $V$ is an interval $|c,d|$. 
If there exists $x\in |c,d|$ such that $x<a$ then any morphism $f:M_{[a,b)}\to V$ must be 0. 
Additionally, if there exists $x\in [a,b)$ such that $x\geq d$ and $x\notin|c,d|$ then any $f:M_{[a,b)}\to V$ must be 0. 
Thus, any morphism $M_{[a,b]}\oplus M_{(a,b)}\to M_{[a,b)}$ must $0$ and morphism $M_{(a,b]}\to M_{[a,b]}\oplus M_{(a,b)}$ must be 0. 

Claim: If $V\not\cong M_{[a,b)}$ and $f:M_{[a,b)}\to V$ is a nonzero morphism then there exists either a nonzero morphism $g_1: M_{[a,b]}\to V$ or $g_2:M_{(a,b)}\to V$ such that $g_i\circ h_i  = f$. 
Proof of claim: If $V\not\cong M_{[a,b)}$ then, by the conditions in the previous paragraph combined with Theorem \ref{thm:iso-indecomps}, either $b\in|c,d|$ or $a\notin |c,d|$. 
If $b\in|c,d|$ Then $g_1$ is a nonzero morphism and so $g_1\circ h_1=f$.
If $a\notin |c,d|$ then $g_2$ is a nonzero morphism and so $g_2\circ h_2=f$.
In either case, $f$ factors through $M_{[a,b]}\oplus M_{(a,b)}$. 

Finally, if $|c,d|=[a,b)$ then by Theorem \ref{thm:iso-indecomps} $f$ is an isomorphism. 
By a dual argument, a morphism from an indecomposable $W$ to $M_{(a,b]}$ that is not an isomorphism factors through $M_{[a,b]}\oplus M_{(a,b)}$. 
Therefore, the given sequence is an Auslander--Reiten sequence.
\end{proof}

We give an example of a representation with no left or right Auslander--Reiten sequences in the form of a proposition.
\begin{proposition}\label{prop:no AR sequences}
Let $M_{\{a\}}$ be the indecomposable representation with support $\{a\}$ where $a$ is neither a sink nor a source.
Then there is are no Auslander--Reiten sequences of either of the following forms:
\begin{displaymath}\xymatrix@R=2ex{ 0\ar[r] & M_{\{a\}} \ar[r] & B\ar[r] & C\ar[r] & 0 \\ 0 \ar[r] & A \ar[r] & B\ar[r] & M_{\{a\}} \ar[r] & 0.}\end{displaymath}
\end{proposition}
\begin{proof}
Suppose $s_{2n}<a<s_{2n+1}$, where $s_{2n}$ is a sink and $s_{2n+1}$ is a source.
The other case is similar. 
For any indecomposable $M_I$, if $\Hom(M_I,M_{\{a\}})\cong k$ then $I=|c,a]$.
We note that, for each $x\in(s_0,a)$, $\Hom(M_{[x,a]},M_{\{a\}})\cong k$ for all $i\geq 0$.

Let $M_I$ be some indecomposable such that $\Hom(M_I,M_{\{a\}})\cong k$.
For any $x\in(s_0,a)$ such that $c<x$ we have $\Hom(M_I,M_{J_x})\cong k$.
Since the Hom space between any two indecomposables is either $k$ or 0 (Proposition \ref{prop:homiskor0}), all nontrivial maps $M_I\to M_{\{a\}}$ factor through every indecomposable $M_{[x,a]}$ for $x\in (s_0,a)$ and $x>c$.
Thus, it is not possible to have an Auslander--Reiten sequence in $\repAR$ of the form $0\to A\to B\to M_{\{a\}}\to 0$.
By a dual argument, the other form is not possible, either.
\end{proof}

\section{Other Papers in this Series}
In Continuous Quivers of Type $A$ (II), the second author
defines a continuous analog of the Auslander--Reiten quiver, called the Auslander--Reiten Space, for both $\repAR$ and its bounded derived category $\mathcal D^b(\repAR)$ \cite{Rock1}.
They show that the Auslander--Reiten space exhibits many of the same properties as an Auslander--Reiten quiver, such as how to find extensions of indecomposables and Auslander--Reiten sequences.
Similar results are shown about the derived category.

The authors define the new continuous cluster category in Continuous Quivers of Type $A$ (III) and generalize cluster structures to cluster theories \cite{IgusaRockTodorov2}.
In particular, they define the $\mathbf{E}$-cluster theory.
They then show many existing type $A$ cluster structures are cluster theories and embed into this new theory in a way that preserves mutation.

In Continuous Quivers of Type $A$ (IV), the second author generalizes mutation to continuous mutation \cite{Rock2}, further generalizing transfinite mutation in \cite{BaurGratz2018}.
The embeddings from Part (III) are shown to be part of a chain of embeddings and the notion of an abstract cluster structure is introduced in order to understand which cluster theories are more strongly related.
Part (IV) concludes with a geometric model of $\mathbf E$-clusters which generalizes the triangulations of polygons and laminations of hyperbolic plane in \cite{CCS} and \cite{IgusaTodorov2015}, respectively.

%

\end{document}